\documentclass[final]{siamltex}

\usepackage{epsfig,amssymb,latexsym}
\usepackage{amsfonts,psfrag,amsmath,bbm,color}
\usepackage{cancel}
\usepackage{comment}
\usepackage{mathrsfs}
\usepackage{graphicx}
\usepackage{textcomp}
\usepackage{multirow}
\usepackage{enumerate}
\usepackage{cancel}
\usepackage{algpseudocode}
\usepackage{caption}  
\usepackage{subcaption}
\usepackage{url}
\usepackage{rotating}
\usepackage{bigints}
\usepackage{hyperref}
\usepackage{xcolor}
\usepackage{ulem}
\usepackage{tabularx}

\usepackage{placeins}

\usepackage[ruled,vlined]{algorithm2e}
\usepackage{geometry}
\vbadness=\maxdimen

\newcommand{\footremember}[2]{%
    \footnote{#2}
    \newcounter{#1}
    \setcounter{#1}{\value{footnote}}%
}
 
\renewcommand{\figurename}{Figure}
\usepackage[most]{tcolorbox}

\newcolumntype{L}{>{\raggedright\arraybackslash}X}

\usepackage{tikz}
\usetikzlibrary{shapes.geometric, arrows}
\tikzstyle{startstop} = [rectangle, rounded corners, minimum width=1cm, minimum height=1cm,text centered, draw=black]
\tikzstyle{io} = [trapezium, trapezium left angle=70, trapezium right angle=110, minimum width=1cm, minimum height=1cm, text centered, draw=black, fill=blue!30]
\tikzstyle{method} = [rectangle, rounded corners, minimum width=1cm, minimum height =1cm, text centered, draw=black]
\tikzstyle{process} = [rectangle, minimum width=1cm, minimum height=1cm, text centered, draw=black]
\tikzstyle{decision} = [diamond, minimum width=0.5cm, minimum height=0.5cm, text centered, draw=black, fill=green!30]
\tikzstyle{arrow} = [thick,->,>=stealth]

\usepackage{amscd}
\usepackage{xcolor}
\graphicspath{{./figs/}}

\newcommand{\lp}{\left(}
\newcommand{\rp}{\right)}

\newtheorem{remark}{Remark}[section]
\newtheorem{assumption}{Assumption}[section]

\def\PP{{{\rm l}\kern - .15em {\rm P} }}
\def\PN2{{\PP_{N}-\PP_{N-2}}}

\newcommand{\cD}{\mathcal{D}}

\newcommand{\bu}{\boldsymbol{u}}

\newcommand{\bx}{\boldsymbol{x}}

\newcommand{\deleted}[1]{{}}

\begin{document}
\title{Decoupled algorithms for non-linearly coupled reaction-diffusion competition model with harvesting and Stocking}

\author{
Muhammad Mohebujjaman\footremember{mit}{D\MakeLowercase{epartment of} M\MakeLowercase{athematics and} P\MakeLowercase{hysics}, T\MakeLowercase{exas} A\&M I\MakeLowercase{nternational} U\MakeLowercase{niversity}, TX 78041, USA; P\MakeLowercase{artially} \MakeLowercase{supported by} N\MakeLowercase{ational} S\MakeLowercase{cience}  F\MakeLowercase{oundation grant} DMS-2213274, \MakeLowercase{and} U\MakeLowercase{niversity} R\MakeLowercase{esearch grant};}\footnote{C\MakeLowercase{orrespondence: m.mohebujjaman@tamiu.edu}}%
\and
Clarisa Buenrostro\footremember{tamiu}{D\MakeLowercase{epartment of} M\MakeLowercase{athematics and} P\MakeLowercase{hysics}, T\MakeLowercase{exas} A\&M I\MakeLowercase{nternational} U\MakeLowercase{niversity}, TX 78041, USA;}
 \and Md. Kamrujjaman\footremember{DU}{D\MakeLowercase{epartment of} M\MakeLowercase{athematics}, U\MakeLowercase{niversity of} D\MakeLowercase{haka}, D\MakeLowercase{haka} 1000, B\MakeLowercase{angladesh};}
 \and Taufiquar Khan \footremember{UNCC}{D\MakeLowercase{epartment of }M\MakeLowercase{athematics and} S\MakeLowercase{tatistics}, U\MakeLowercase{niversity of} N\MakeLowercase{orth} C\MakeLowercase{arolina at} C\MakeLowercase{harlotte}, NC 28223, USA;}
 }

\maketitle



\pagestyle{myheadings}
\thispagestyle{plain}

\markboth{\MakeUppercase{Decoupled algorithms for reaction-diffusion competition model}}{\MakeUppercase{ M. Mohebujjaman, C. Buenrostro, M. Kamrujjaman, and T. Khan}}

\begin{abstract}
    We propose, analyze and test two novel fully discrete decoupled linearized algorithms for a nonlinearly coupled reaction-diffusion $N$-species competition model with harvesting or stocking effort. The time-stepping algorithms are first and second order accurate in time and optimally accurate in space. Stability and optimal convergence theorems of the decoupled schemes are proven rigorously. We verify the predicted convergence rates of our analysis and efficacy of the algorithms using numerical experiments and synthetic data for analytical test problems. We also study the effect of harvesting or stocking and diffusion parameters on the evolution of species population density numerically, and observe the co-existence scenario subject to optimal harvesting or stocking. 
\end{abstract}

\section{Introduction} 
 In an environmental approach, one of the most significant concerns in population dynamics is the effect of harvesting or stocking which plays a crucial role for optimal management of limited resources to ensure the balance in ecology  \cite{clark1974mathematical,dai1998coexistence,li2011optimal,otunuga2021time,yang2019optimal,zhang2003optimal}. Harvesting
indicates reducing the population size due to hunting, fishing, or capturing, which shrinks the population density. The study of harvesting for one population was limited in \cite{braverman2009optimal,korobenko2013persistence, roques2007population}, and in some
situations, these are unable to explain the actual situation better. 
More interesting situations are discovered when harvesting is implied for two or more interacting population dynamics \cite{clayton1997bringing,leung1995optimal,liu2016optimal, stigter2004optimal} that represent either coexistence or competitive exclusion by others. A global behaviour of predator-prey systems is analyzed under constant harvesting or stocking of either or both species in \cite{brauer1982constant}. To present the pattern and visualize the effects of harvesting, reaction-diffusion equation is the constitutive equation of population dynamics, e.g., competition and prey-predator models. 

We consider the efficient and accurate numerical approximation of the population dynamics in a reaction-diffusion $N$-species competition model with harvesting or stocking, which is governed by the following system of nonlinear evolutionary equations \cite{adan2022role,braverman2019interplay, wong2009analysis}: For $i=1,2,\cdots,N$
\begin{align}
    \frac{\partial u_i}{\partial t}&=d_i\Delta u_i+r_iu_i\left(1-\mu_i-\frac{1}{K}\sum\limits_{j=1}^N u_j\right)+f_i, \hspace{2mm}\forall (t,\bx)\in (0,T]\times\Omega,\label{RDE1}
\end{align}
together with known initial and boundary conditions (which are suppressed momentarily)
where $u_i$, $d_i$, $r_i$ and $\mu_i$ represent the population density, diffusion rate, intrinsic growth rate and harvesting or stocking coefficient of the $i^{th}$ competing species, respectively. $N$ denotes the number of species in the competition, if $N=1$, the model \eqref{RDE1} represents simple logistic growth model of a single species. Here, $K$ represents carrying capacity of the heterogeneous environment, $f_i$ the forcing, $t$ the time, $\bx\in\Omega$ the spatial variable, $\Omega$ the domain, and $T$ the simulation end time. It is assumed  that the harvesting rate is proportional to the intrinsic growth rate in the model \eqref{RDE1}.

The difficultly in simulating equation \eqref{RDE1} is that we need to solve a non-linearly coupled system of partial differential equations at each time-step, where the intrinsic growth rates and carrying capacity all depend on space and time. It is an open problem how to decouple the system in a stable way. A three-species competition-diffusion model with constant intrinsic growth rate in a homogeneous environment ($K\equiv$ constant) without harvesting or stocking is given by Wong in \cite{wong2009analysis}. The author presented one first-order and another second-order decoupled time-stepping discrete schemes and their convergence rates however only the first-order scheme in a finite element setting was analyzed, and no numerical experiments were given beyond the convergence rate verification. The optimal harvesting in controlling species density in a two-species competition model with a heterogeneous environment is investigated in \cite{adan2022role}, where  a fully-discrete backward-Euler decoupled time-stepping algorithm is used without any analysis of the discrete algorithm. A Lotka–Volterra interactions model with no-flux boundary conditions in the presence of prey-taxis and spatial diffusion is given in \cite{ainseba2008reaction} and discussed the existence and uniqueness of the weak solution. Kamrujjaman et al. studied the spatial-temporal effects for logistic and Gilpin-Ayala growth function with starvation type diffusion for single species population with stocking \cite{kamrujjaman2022spatio, Zahan2022Mathmatical}. They studied the stability properties for the existence and extinction of species. Also, in the case of space-dependent carrying capacity, they established the presence of optimal harvesting efforts. They presented their outcomes analytically and computationally. The main interest focused on the analytical approach instead of claiming any robust numerical algorithm.
\subsection{Significance of the work}
In this paper, we propose, analyze, and test two fully discrete and decoupled linearized stable time-stepping algorithms of a non-linearly coupled system of reaction-diffusion equations that describes an $N$-species competition model in a heterogeneous environment with harvesting or stocking. We provide rigorous analysis of the existence and uniqueness of the solutions of the algorithms together with the priori error estimates by proving their stability and convergence theorems. We prove that the both algorithms are optimally accurate in time and space. The numerical tests are presented showing their convergence rates on some known analytical test problems varying number of species. The solution at each time-step can be computed simultaneously for each species in the competition, which can reduce a huge computational cost when compared to coupled non-linear algorithms. A series of numerical experiments are given that show the effect of exponentially varying carrying capacity, non-stationary intrinsic growth rates, varying diffusion parameters, and harvesting or stocking on the population density of the species in the competition.

To the best of our knowledge, the proposed efficient fully-discrete algorithms of the $N$-species reaction-diffusion competition model in \eqref{RDE1} with harvesting or stocking have not been investigated to date. The proposed algorithms are expected to enable new tools for large-scale computing in population dynamics.

The rest of the paper is organized as follows: In Section \ref{notation-preliminaries}, we present some necessary notation and preliminaries for a thorough analysis. We present two fully discrete decoupled schemes and analyze them in Section \ref{fully-discrete-schme}. In Section \ref{numerical-experiments}, we perform several numerical experiments to support the theoretical findings in Section \ref{fully-discrete-schme}. Finally, the conclusion and discussions of future research are given in Section \ref{conclusion}.

\section{Notation and preliminaries} \label{notation-preliminaries} Let $\Omega\subset\mathbb{R}^d (d\in\{1,2,3\})$ be a  convex domain with boundary $\partial\Omega$. For a given carrying capacity $K:(0,T]\times\Omega\rightarrow\mathbb{R}$, we define
\begin{align}
    K_{\min}:=\inf\limits_{(t,\bx)\in (0,T]\times\Omega}|K(t,\bx)|,\label{kmin-def}
    \end{align}
    and assume $K_{\min}>0$. The usual $L^2(\Omega)$ norm and inner product are denoted by $\|.\|$ and $(.,.)$, respectively. Similarly, the $L^p(\Omega)$ norms and the Sobolev $W_p^k(\Omega)$ norms are $\|.\|_{L^p}$ and $\|.\|_{W_p^k}$, respectively for $k\in\mathbb{N},\hspace{1mm}1\le p\le \infty$. The Sobolev space $W_2^k(\Omega)^d$ is represented by $H^k(\Omega)^d$ with norm $\|.\|_k$ which are Hilbert spaces.

For $X$ being a normed function space in $\Omega$, $L^p(0,T;X)$ is the space of all functions defined on $(0,T]\times\Omega$ for which the following norm 
\begin{align*}
    \|u\|_{L^p(0,T;X)}=\lp\int_0^T\|u\|_{X}^pdt\rp^\frac{1}{p},\hspace{2mm}p\in[1,\infty)
\end{align*}
is finite. For $p=\infty$, the usual modification is used in the definition of this space. We denote $$\|u\|_{\infty,\infty}:=\|u\|_{L^\infty\big(0,T;L^\infty(\Omega)^d\big)}.$$ The natural function spaces for our problem are
\begin{align*}
    X:&=H_0^1(\Omega)=\big\{v\in L^2(\Omega) :\nabla v\in L^2(\Omega)^{d},\hspace{1mm}  v=0 \hspace{1mm} \mbox{on}\hspace{1mm}   \partial \Omega\big\}.
\end{align*}
For an element $f$ in the dual space of $X$, the norm is defined by
$$\|f\|_{-1}:=\sup\limits_{v\in X}\frac{(f,v)}{\|\nabla v\|}.$$
Recall the Poincar\'e inequality holds in $X$: There exists $C$ depending only on $\Omega$ satisfying for all $\phi\in X$,
\[
\| \phi \| \le C \| \nabla \phi \|.
\]
Multiplying both sides of \eqref{RDE1} by $v_i\in X$ and integrating over $\Omega$, we have the following continuous weak form: For $i=1,2,\cdots,N$
\begin{align}
    \left(\frac{\partial u_i}{\partial t},v_i\right)+d_i\left(\nabla u_i,\nabla v_i\right)=(1-\mu_i)\big(r_i(t,\bx)u_i,v_i\big)-\left(\frac{r_i(t,\bx) u_i}{K(t,\bx)}\sum\limits_{j=1}^Nu_j,v_i\right)+\left(f_i,v_i\right).\label{vec-weak-form}
\end{align}
The conforming finite element space is denoted by $X_h\subset X$, and we assume a sufficiently regular triangulation $\tau_h(\Omega)$ for the inverse inequality to hold, where $h$ is the maximum triangle diameter. We have the following approximation properties typical of piecewise polynomials of degree $k$ in $X_h$: \cite{BS08,linke2019pressure}
\begin{align}
\|u- P^{L^2}_{X_h}(u) \|&\leq Ch^{k+1}|u|_{k+1},\hspace{2mm}u\in H^{k+1}(\Omega),\label{AppPro3}\\
\| \nabla (u- P^{L^2}_{X_h}(u)  ) \|&\leq Ch^{k}|u|_{k+1},\hspace{2mm}u\in H^{k+1}(\Omega),\label{AppPro4}
\end{align}
where $P^{L^2}_{X_h}(u)$ is the $L^2$ projection of $u$ into $X_h$ and $|\cdot|_r$ denotes the $H^r$ seminorm. Note that $C>0$ is a generic constant and changes in computation. 
\textcolor{black}{The following lemma for the discrete Gr\"onwall inequality was given in \cite{HR90}.
\begin{lemma}\label{dgl} Let $\mathbb{N}$ denotes the set of all natural numbers and
 $\Delta t$, $\cD$, $a_n$, $b_n$, $c_n$, $d_n$ be non-negative numbers for $n=1,\cdots\hspace{-0.35mm},M$ such that
    $$a_M+\Delta t \sum_{n=1}^Mb_n\leq \Delta t\sum_{n=1}^{M-1}{d_na_n}+\Delta t\sum_{n=1}^Mc_n+\cD\hspace{3mm}\mbox{for}\hspace{2mm}M\in\mathbb{N},$$
then for all $\Delta t> 0,$
$$a_M+\Delta t\sum_{n=1}^Mb_n\leq \lp\Delta t\sum_{n=1}^Mc_n+\cD\rp\mbox{exp}\left(\Delta t\sum_{n=1}^{M-1}d_n\right)\hspace{2mm}\mbox{for}\hspace{2mm}M\in\mathbb{N}.$$
\end{lemma}}

\section{Fully discrete scheme} \label{fully-discrete-schme}
In this section, we propose and analyze two fully discrete, decoupled, and linearized time-stepping algorithms for approximating a solution of \eqref{RDE1}. The Decoupled Backward-Euler (DBE) scheme is presented in Algorithm \ref{Algn1}, which approximates the temporal derivative by first-order backward-Euler formula and the non-linear term is linearized by the immediate previous time-step solution. In Algorithm \ref{Algn2}, we present Decoupled Backward Difference Formula 2 (DBDF-2) scheme, which consists of second-order accurate time derivative approximation formula, and linearizes the non-linear term by second-order approximation of the unknown solution at the previous time-step. 

 \begin{algorithm}[H]\label{Algn1}
  \caption{DBE scheme} Given time-step $\Delta t>0$, end time $T>0$, for $i=1,2,\cdots,N$, initial conditions $u_i^0\in L^2(\Omega)^d$, $f_i\in L^\infty\left(0,T;H^{-1}(\Omega)^d\right)$, $K_{\min}>0$, and $r_i\in L^\infty(0,T;L^\infty(\Omega)^d)$. Set $M=T/\Delta t$ and for $n=0,1,\cdots\hspace{-0.35mm},M-1$, compute:
 Find $u_{i,h}^{n+1}\in X_h$ satisfying,  $\forall v_{i,h}\in X_h$:
 \begin{align}
    \left(\frac{u_{i,h}^{n+1}-u_{i,h}^{n}}{\Delta  t},v_{i,h}\right)&+d_i\left(\nabla u_{i,h}^{n+1},\nabla v_{i,h}\right)=(1-\mu_i)\left(r_i(t^{n+1})u_{i,h}^{n+1},v_{i,h}\right)\nonumber\\&-\left(\frac{r_i(t^{n+1}) u_{i,h}^{n+1}}{K(t^{n+1})}\sum\limits_{j=1}^Nu_{j,h}^n,v_{i,h}\right)+\left(f_i(t^{n+1}),v_{i,h}\right).\label{disc-weak-form}
\end{align}
\end{algorithm}

\begin{algorithm}[H]\label{Algn2}
  \caption{DBDF-2 scheme} Given time-step $\Delta t>0$, end time $T>0$, for $i=1,2,\cdots,N$, initial conditions $u_i^0,\;u_i^1\in L^2(\Omega)^d$, $f_i\in L^\infty\left(0,T;H^{-1}(\Omega)^d\right)$, $K_{\min}>0$, and $r_i\in L^\infty(0,T;L^\infty(\Omega)^d)$. Set $M=T/\Delta t$ and for $n=1,\cdots\hspace{-0.35mm},M-1$, compute:
 Find $u_{i,h}^{n+1}\in X_h$ satisfying,  $\forall v_{i,h}\in X_h$:
 \begin{align}
    \Bigg(\frac{3u_{i,h}^{n+1}-4u_{i,h}^{n}+u_{i,h}^{n-1}}{2\Delta  t},&v_{i,h}\Bigg)+d_i\left(\nabla u_{i,h}^{n+1},\nabla v_{i,h}\right)=(1-\mu_i)\left(r_i(t^{n+1})u_{i,h}^{n+1},v_{i,h}\right)\nonumber\\&-\left(\frac{r_i(t^{n+1}) u_{i,h}^{n+1}}{K(t^{n+1})}\sum\limits_{j=1}^N(2u_{j,h}^n-u_{j,h}^{n-1}),v_{i,h}\right)+\left(f_i(t^{n+1}),v_{i,h}\right).\label{disc-weak-form2}
\end{align}
\end{algorithm}
These types of splitting algorithms are commonly used in magnetohydrodynamics \cite{AKMR15,HMR17,MR17}. Throughout the analysis of this paper, we will consider the following assumption:

\begin{assumption}
Let's assume that there exists a constant $C>0$ such that $\|u_{i,h}^n\|_{\infty}\le C$ for $i=1,2,\cdots,N$.\label{assumption-1}
\end{assumption}

We will prove the Assumption \ref{assumption-1} holds true at the end of Section \ref{Convergence-sec} in Lemma \ref{lemma-discrete-bound}.

\subsection{Stability analysis}\label{stability-analysis}

In this section, we prove the stability theorems and well-posedness of DBE and DBDF-2 schemes. For simplicity of our analysis, we define
\begin{align}
    \alpha_i:=d_i-C\|r_i\|_{\infty,\infty}\left(|1-\mu_i|+\frac{1}{K_{\min}}\right),\label{alpha-def}
\end{align}
 for $i=1,2,\cdots,N$.

\begin{theorem}(Stability of DBE)\label{stability-theorm}
For $i=1,2,\cdots,N$, 
assume $u_{i,h}^0\in H^1(\Omega)^d$, $f_i\in L^\infty\left(0,T;H^{-1}(\Omega)^d\right)$, $r_i\in L^\infty(0,T;L^\infty(\Omega)^d)$, $K_{\min}>0$ and under the Assumption \ref{assumption-1}, if $\alpha_i> 0$, then for any $\Delta t>0$
\begin{align}
   \|u_{i,h}^M\|^2+2\alpha_i\Delta t\sum\limits_{n=1}^M\|\nabla u_{i,h}^n\|^2\le\|u_{i,h}^0\|^2+\frac{\Delta t}{\alpha_i}\sum\limits_{n=1}^{M} \|f_i(t^{n})\|_{-1}.
\end{align}
\end{theorem}

\begin{proof}
Taking $v_{i,h}=u_{i,h}^{n+1}$ in \eqref{disc-weak-form}, and using the polarization identity
$$(b-a,b)=\frac12\left(\|b-a\|^2+\|b\|^2-\|a\|^2\right),$$
gives
\begin{align}
    \frac{1}{2\Delta t}\Big(\|u_{i,h}^{n+1}-u_{i,h}^{n}\|^2+\|u_{i,h}^{n+1}\|^2&-\|u_{i,h}^{n}\|^2\Big)+d_i\|\nabla u_{i,h}^{n+1}\|^2=(1-\mu_i)\left(r_i(t^{n+1})u_{i,h}^{n+1},u_{i,h}^{n+1}\right)\nonumber\\&-\left(\frac{r_i(t^{n+1}) u_{i,h}^{n+1}}{K(t^{n+1})}\sum\limits_{j=1}^Nu_{j,h}^n,u_{i,h}^{n+1}\right)+\left(f_i(t^{n+1}),u_{i,h}^{n+1}\right).\label{pol-new}
\end{align}
We apply H\"{o}lder's inequality on the first two terms and Cauchy Schwarz's inequality on the forcing term on the right-hand-side of \eqref{pol-new}, we have
\begin{align}
   \frac{1}{2\Delta t}\Big(\|u_{i,h}^{n+1}-u_{i,h}^{n}\|^2+&\|u_{i,h}^{n+1}\|^2-\|u_{i,h}^{n}\|^2\Big)+d_i\|\nabla u_{i,h}^{n+1}\|^2 \le|1-\mu_i|\|r_i(t^{n+1})\|_{\infty}\|u_{i,h}^{n+1}\|^2\nonumber\\&+\Big\|\frac{r_i(t^{n+1})}{K(t^{n+1})}\Big\|_\infty\sum\limits_{J=1}^N\|u_{i,h}^{n+1}\|^2\|u_{j,h}^n\|_{\infty}+\|f_i(t^{n+1})\|_{-1}\|\nabla u_{i,h}^{n+1}\|.\label{before-small-data-assumption}
\end{align}
Using Poincar\'e inequality and the Assumption \ref{assumption-1}, we have
\begin{align}
   \frac{1}{2\Delta t}\big(\|u_{i,h}^{n+1}-u_{i,h}^{n}\|^2+\|u_{i,h}^{n+1}\|^2&-\|u_{i,h}^{n}\|^2\big)+d_i\|\nabla u_{i,h}^{n+1}\|^2\le C|1-\mu_i|\|r_i\|_{L^\infty\big(0,T;L^\infty(\Omega)^d\big)}\|\nabla u_{i,h}^{n+1}\|^2\nonumber\\&+\frac{C\|r_i\|_{L^\infty\big(0,T;L^\infty(\Omega)^d\big)}}{\inf\limits_{(t,\bx)\in (0,T]\times\Omega}|K|}\|\nabla u_{i,h}^{n+1}\|^2+\|f_i(t^{n+1})\|_{-1}\|\nabla u_{i,h}^{n+1}\|.
\end{align}
Grouping terms on the left-hand-side and using \eqref{alpha-def}, and \eqref{kmin-def}, yields
\begin{align}
   \frac{1}{2\Delta t}&\left(\|u_{i,h}^{n+1}-u_{i,h}^{n}\|^2+\|u_{i,h}^{n+1}\|^2-\|u_{i,h}^{n}\|^2\right)+\alpha_i\|\nabla u_{i,h}^{n+1}\|^2 \le \|f_i(t^{n+1})\|_{-1}\|\nabla u_{i,h}^{n+1}\|.
\end{align}
Assume $\alpha_i>0$, use Young's inequality, and hide term on left-hand-side to obtain
\begin{align}
   \frac{1}{2\Delta t}&\left(\|u_{i,h}^{n+1}-u_{i,h}^{n}\|^2+\|u_{i,h}^{n+1}\|^2-\|u_{i,h}^{n}\|^2\right)+\frac{\alpha_i}{2}\|\nabla u_{i,h}^{n+1}\|^2 \le\frac{1}{2\alpha_i} \|f_i(t^{n+1})\|_{-1}.
\end{align}
Now, multiply both sides by $2\Delta t$, and sum over time steps from $n=0,1,\cdots, M-1$, we have
\begin{align}
    \|u_{i,h}^M\|^2+\sum\limits_{n=0}^{M-1}\|\bu_{i,h}^{n+1}-\bu_{i,h}^{n}\|^2+\alpha_i\Delta t\sum\limits_{n=1}^M\|\nabla u_{i,h}^n\|^2\le\|u_{i,h}^0\|^2+\frac{\Delta t}{\alpha_i}\sum\limits_{n=0}^{M-1} \|f_i(t^{n+1})\|_{-1}.
\end{align}
Now, dropping non-negative terms from left-hand-side completes the proof.
\end{proof}

\begin{theorem}(Stability of DBDF-2)\label{stability-theorm2}
For $i=1,2,\cdots,N$, 
assume $u_{i,h}^0,u_{i,h}^1\in L^2(\Omega)^d$, $f_i\in L^\infty\left(0,T;H^{-1}(\Omega)^d\right)$, $K_{\min}>0$, $r_i\in L^\infty(0,T;L^\infty(\Omega)^d)$, and under the Assumption \ref{assumption-1}, if $\alpha_i> 0$, then for any $\Delta t>0$
\begin{align}
       \|u_{i,h}^{M}\|^2+\|2u_{i,h}^{M}-u_{i,h}^{M-1}\|^2+2\alpha_i\Delta t\sum\limits_{n=2}^{M}\|\nabla u_{i,h}^{n}\|^2\le\|u_{i,h}^{1}\|^2+\|2u_{i,h}^{1}-u_{i,h}^{0}\|^2+\frac{2\Delta t}{\alpha_i}\sum\limits_{n=2}^{M}\|f_i(t^{n})\|_{-1}.
   \end{align}
\end{theorem}

\begin{proof}
Taking $v_{i,h}=u_{i,h}^{n+1}$ in \eqref{disc-weak-form2} to obtain
\begin{align}
    \Bigg(\frac{3u_{i,h}^{n+1}-4u_{i,h}^{n}+u_{i,h}^{n-1}}{2\Delta  t},&u_{i,h}^{n+1}\Bigg)+d_i\|\nabla u_{i,h}^{n+1}\|^2=(1-\mu_i)\left(r_i(t^{n+1})u_{i,h}^{n+1},u_{i,h}^{n+1}\right)\nonumber\\&-\left(\frac{r_i(t^{n+1}) u_{i,h}^{n+1}}{K(t^{n+1})}\sum\limits_{j=1}^N(2u_{j,h}^n-u_{j,h}^{n-1}),u_{i,h}^{n+1}\right)+\left(f_i(t^{n+1}),u_{i,h}^{n+1}\right).
\end{align}
Using the following identity
\begin{eqnarray}
(3a-4b+c,a)=\frac{a^2+(2a-b)^2}{2}-\frac{b^2+(2b-c)^2}{2}+\frac{(a-2b+c)^2}{2},\label{ident}
\end{eqnarray}
we write
\begin{align}
    \frac{1}{4\Delta t}\Big(\|u_{i,h}^{n+1}\|^2-\|u_{i,h}^{n}\|^2+\|2u_{i,h}^{n+1}-u_{i,h}^{n}\|^2-\|2u_{i,h}^{n}-u_{i,h}^{n-1}\|^2+\|u_{i,h}^{n+1}-2u_{i,h}^{n}+u_{i,h}^{n-1}\|^2\Big)\nonumber\\+d_i\|\nabla u_{i,h}^{n+1}\|^2=(1-\mu_i)\left(r_i(t^{n+1})u_{i,h}^{n+1},u_{i,h}^{n+1}\right)-\left(\frac{r_i(t^{n+1}) u_{i,h}^{n+1}}{K(t^{n+1})}\sum\limits_{j=1}^N(2u_{j,h}^n-u_{j,h}^{n-1}),u_{i,h}^{n+1}\right)\nonumber\\+\left(f_i(t^{n+1}),u_{i,h}^{n+1}\right).\label{BDF2-identity}
\end{align}
We apply H\"{o}lder's inequality on the first two terms and Cauchy Schwarz's inequality on the forcing term on the right-hand-side of \eqref{BDF2-identity}, we have
\begin{align}
   \frac{1}{4\Delta t}\Big(\|u_{i,h}^{n+1}\|^2-\|u_{i,h}^{n}\|^2+\|2u_{i,h}^{n+1}-u_{i,h}^{n}\|^2-\|2u_{i,h}^{n}-u_{i,h}^{n-1}\|^2+\|u_{i,h}^{n+1}-2u_{i,h}^{n}+u_{i,h}^{n-1}\|^2\Big)\nonumber\\+d_i\|\nabla u_{i,h}^{n+1}\|^2 \le|1-\mu_i|\|r_i(t^{n+1})\|_{\infty}\|u_{i,h}^{n+1}\|^2\nonumber\\+\Big\|\frac{r_i(t^{n+1})}{K(t^{n+1})}\Big\|_\infty\sum\limits_{J=1}^N\|u_{i,h}^{n+1}\|^2\big(2\|u_{j,h}^n\|_{\infty}+\|u_{j,h}^{n-1}\|_{\infty}\big)+\|f_i(t^{n+1})\|_{-1}\|\nabla u_{i,h}^{n+1}\|.\label{before-small-data-assumption-bdf2}
   \end{align}
   Using Poincar\'e inequality, the Assumption \ref{assumption-1}, and grouping terms on the left-hand-side to obtain
   \begin{align}
   \frac{1}{4\Delta t}\Big(\|u_{i,h}^{n+1}\|^2-\|u_{i,h}^{n}\|^2+\|2u_{i,h}^{n+1}-u_{i,h}^{n}\|^2-\|2u_{i,h}^{n}-u_{i,h}^{n-1}\|^2+\|u_{i,h}^{n+1}-2u_{i,h}^{n}+u_{i,h}^{n-1}\|^2\Big)\nonumber\\+\alpha_i\|\nabla u_{i,h}^{n+1}\|^2 \le \|f_i(t^{n+1})\|_{-1}\|\nabla u_{i,h}^{n+1}\|.
   \end{align}
   Drop non-negative term from left-hand-side, assume $\alpha_i>0$, use Young's inequality, and hide term on left-hand-side to obtain
   \begin{align}
   \frac{1}{4\Delta t}\Big(\|u_{i,h}^{n+1}\|^2-\|u_{i,h}^{n}\|^2+\|2u_{i,h}^{n+1}-u_{i,h}^{n}\|^2-\|2u_{i,h}^{n}-u_{i,h}^{n-1}\|^2\Big)\nonumber\\+\frac{\alpha_i}{2}\|\nabla u_{i,h}^{n+1}\|^2 \le\frac{1}{2\alpha_i} \|f_i(t^{n+1})\|_{-1}.
   \end{align}
   Now, multiply both sides by $4\Delta t$, and sum over time-steps from $n=1,\cdots, M-1$ finishes the proof.
\end{proof}

\begin{remark}
The finite dimensional schemes Algorithm \ref{Algn1} and \ref{Algn2} are linear at each time-step and the stability theorems provide their solutions are bounded continuously by the problem data, which is sufficient for the well-posedness of the schemes. The linearity of the schemes provides the uniqueness of the solution via their the stability theorem. Because of the finite dimensional and linearity features, the uniqueness implies existence of the solution, therefore the solution to the Algorithm \ref{Algn1} and \ref{Algn2} exist uniquely 
\end{remark}\cite{HMR17,mohebujjaman2022efficient}.

\subsection{Convergence analysis}\label{Convergence-sec}

In this section, we will provide apriori estimates of the errors in the computed species density using the both DBE and DBDF-2 schemes.

\begin{theorem}(Error estimate of DBE) Consider $m=\max\{2,k+1\}$, and $i=1,2,\cdots,N$, assume $u_i$ solves \eqref{RDE1} and satisfies
\begin{align*}
    u_i\in L^\infty\big(0,T;H^m(\Omega)^d\big),&\hspace{1mm}u_{i,t}\in L^\infty\big(0,T;L^2(\Omega)^d\big),\hspace{1mm}u_{i,tt}\in L^\infty\big(0,T;L^2(\Omega)^d\big),\\&r_i\in L^\infty\left(0,T;L^\infty(\Omega)^d\right),\text{ and } K_{\min}>0,
\end{align*}

if $\alpha_i>0$ then for $\Delta t>0$ the solution $u_{i,h}$ to the Algorithm  \ref{Algn1} converges to the true solution with
\begin{align}
    \sum\limits_{i=1}^N\|u_i(T)-u_{i,h}^M\|+\sum\limits_{i=1}^N\left\{\alpha_i\Delta t\sum_{n=1}^M\|\nabla \big(u_i(t^n)-u_{i,h}^n\big)\|^2\right\}^{\frac{1}{2}}\le C \big(h^{k}+\Delta t\big).
\end{align}\label{Convergence-analysis-BE}
\end{theorem}

\begin{proof}
At first we build an error equation at the time level $t^{n+1}$, the continuous variational formulations can be written as $\forall v_{i,h}\in X_h$
\begin{align}
    \Bigg(&\frac{u_{i}(t^{n+1})-u_{i}(t^{n})}{\Delta  t},v_{i,h}\Bigg)+d_i\left(\nabla u_{i}(t^{n+1}),\nabla v_{i,h}\right)=(1-\mu_i)\left(r_i(t^{n+1})u_{i}(t^{n+1}),v_{i,h}\right)+\left(f_i(t^{n+1}),v_{i,h}\right)\nonumber\\&-\left(\frac{r_i(t^{n+1}) u_{i}(t^{n+1})}{K(t^{n+1})}\sum\limits_{j=1}^Nu_{j}(t^{n+1}),v_{i,h}\right)+\left(\frac{u_{i}(t^{n+1})-u_{i}(t^{n})}{\Delta  t}- u_{i,t}(t^{n+1}),v_{i,h}\right).\label{cont-weak-form}
\end{align}
Denote $e_i^n:=u_i(t^{n+1})-u_{i,h}^n$. Subtract \eqref{disc-weak-form} from \eqref{cont-weak-form} and then rearranging yields
\begin{align}
    \left(\frac{e_{i}^{n+1}-e_{i}^{n}}{\Delta  t},v_{i,h}\right)+d_i\left(\nabla e_{i}^{n+1},\nabla v_{i,h}\right)-(1-\mu_i)\left(r_i(t^{n+1})e_{i}^{n+1},v_{i,h}\right)\nonumber\\+\sum\limits_{j=1}^N\left(\frac{r_i(t^{n+1})}{K(t^{n+1})}e_i^{n+1}u_{j,h}^n,v_{i,h}
    \right)=G(t,u_i,v_{i,h}),\label{error-equation}
\end{align}
where $$G(t,u_i,v_{i,h})=\left(\frac{u_{i}(t^{n+1})-u_{i}(t^{n})}{\Delta  t}- u_{i,t}(t^{n+1}),v_{i,h}\right)+\left(\frac{r_i(t^{n+1}) u_{i}(t^{n+1})}{K(t^{n+1})}\sum\limits_{j=1}^N\big\{u_{j,h}^{n}-u_{j}(t^{n+1})\big\},v_{i,h}\right).$$
Now we decompose the errors as
\begin{align*}
    e_{i}^n:& = u_i(t^n)-u_{i,h}^n=(u_i(t^n)-\tilde{u}_i^n)-(u_{i,h}^n-\tilde{u}_i^n):=\eta_{i}^n-\phi_{i,h}^n,
\end{align*}
where $\tilde{u}_i^n: =P_{X_h}^{L^2}(u_i(t^n))\in X_h$ is the $L^2$ projections of $u_j(t^n)$ into $X_h$. Note that $(\eta_{i}^n,v_{i,h})=0\hspace{2mm} \forall v_{i,h}\in X_h$.  Rewriting, we have for $v_{i,h}\in X_h$
\begin{align}
    \Bigg(&\frac{\phi_{i,h}^{n+1}-\phi_{i,h}^{n}}{\Delta  t},v_{i,h}\Bigg)+d_i\left(\nabla \phi_{i,h}^{n+1},\nabla v_{i,h}\right)-(1-\mu_i)\left(r_i(t^{n+1})\phi_{i,h}^{n+1},v_{i,h}\right)\nonumber\\&+\sum\limits_{j=1}^N\left(\frac{r_i(t^{n+1})}{K(t^{n+1})}\phi_{i,h}^{n+1}u_{j,h}^n,v_{i,h}
    \right)=d_i\left(\nabla \eta_{i}^{n+1},\nabla v_{i,h}\right)-(1-\mu_i)\left(r_i(t^{n+1})\eta_{i}^{n+1},v_{i,h}\right)\nonumber\\&+\sum\limits_{j=1}^N\left(\frac{r_i(t^{n+1})}{K(t^{n+1})}\eta_i^{n+1}u_{j,h}^n,v_{i,h}
    \right)-G(t,u_i,v_{i,h}).\label{phi-equation}
\end{align}
Choose $v_{i,h}=\phi_{i,h}^{n+1}$, and use the polarization identity in \eqref{phi-equation}, to obtain
\begin{align}
    \frac{1}{2\Delta t}\left(\|\phi_{i,h}^{n+1}-\phi_{i,h}^{n}\|^2+\|\phi_{i,h}^{n+1}\|^2-\|\phi_{i,h}^{n}\|^2\right)+d_i\|\nabla\phi_{i,h}^{n+1}\|^2-(1-\mu_i)\left(r_i(t^{n+1})\phi_{i,h}^{n+1},\phi_{i,h}^{n+1}\right)\nonumber\\+\sum\limits_{j=1}^N\left( \frac{r_i(t^{n+1}) }{K(t^{n+1})}\phi_{i,h}^{n+1}u_{j,h}^n,\phi_{i,h}^{n+1}\right)=d_i\left(\nabla \eta_{i}^{n+1},\nabla \phi_{i,h}^{n+1}\right)-(1-\mu_i)\left(r_i(t^{n+1})\eta_{i}^{n+1},\phi_{i,h}^{n+1}\right)\nonumber\\+\sum\limits_{j=1}^N\left(\frac{r_i(t^{n+1})}{K(t^{n+1})}\eta_{i}^{n+1} u_{j,h}^n,\phi_{i,h}^{n+1}\right)-G\left(t,u_i,\phi_{i,h}^{n+1}\right).\label{all-phi}
\end{align}
Now, we find the upper-bounds of terms in the above equation. Using H\"older's, and Poincar\'e inequalities, we have
\begin{align*}
    (1-\mu_i)\left(r_i(t^{n+1})\phi_{i,h}^{n+1},\phi_{i,h}^{n+1}\right)&\le|1-\mu_i|\|r_i(t^{n+1})\|_{\infty}\| \phi_{i,h}^{n+1}\|^2\\&\le C |1-\mu_i|\|r_i\|_{\infty,\infty}\|\nabla \phi_{i,h}^{n+1}\|^2.
\end{align*}
Next, using triangle, H\"older's, and Poincar\'e inequalities together with the Assumption \ref{assumption-1}, we have
\begin{align*}
    -\sum\limits_{j=1}^N\left( \frac{r_i(t^{n+1}) }{K(t^{n+1})}\phi_{i,h}^{n+1}u_{j,h}^n,\phi_{i,h}^{n+1}\right)&\le\sum\limits_{j=1}^N\frac{\|r_i(t^{n+1})\|_{\infty}}{\inf\limits_\Omega\|K(t^{n+1})\|}\Big|\left(\phi_{i,h}^{n+1}u_{j,h}^n,\phi_{i,h}^{n+1}\right)\Big|\nonumber\\&\le\sum\limits_{j=1}^N\frac{\|r_i\|_{\infty,\infty}}{K_{\min}}\|u_{j,h}^{n}\|_\infty\|\phi_{i,h}^{n+1}\|^2\nonumber\\&\le\frac{C\|r_i\|_{\infty,\infty}}{K_{\min}}
    \|\nabla\phi_{i,h}^{n+1}\|^2.
\end{align*} With the assumption $\alpha_i>0$, use Cauchy-Schwarz, and Young's inequalities, to obtain
\begin{align*}
    d_i\left(\nabla \eta_{i}^{n+1},\nabla \phi_{i,h}^{n+1}\right)\le d_i\|\nabla \eta_{i}^{n+1}\|\|\nabla \phi_{i,h}^{n+1}\|\le\frac{\alpha_i}{10}\|\nabla \phi_{i,h}^{n+1}\|^2+\frac{5d_i^2}{2\alpha_i}\|\nabla \eta_{i}^{n+1}\|^2.
\end{align*}
Using H\"older's,  Poincar\'e, and Young's inequalities, we have
\begin{align*}
    (1-\mu_i)\left(r_i(t^{n+1})\eta_{i}^{n+1},\phi_{i,h}^{n+1}\right)&\le|1-\mu_i|\|r_i(t^{n+1})\|_\infty\|\eta_{i}^{n+1}\|\|\phi_{i,h}^{n+1}\|\\&\le C|1-\mu_i|\|r_i\|_{\infty,\infty}\|\eta_{i}^{n+1}\|\|\nabla\phi_{i,h}^{n+1}\|\\&\le\frac{\alpha_i}{10}\|\nabla\phi_{i,h}^{n+1}\|^2+\frac{C(1-\mu_i)^2\|r_i\|_{\infty,\infty}^2}{\alpha_i}\|\eta_{i}^{n+1}\|^2,\\
    \sum\limits_{j=1}^N\left(\frac{r_i(t^{n+1})}{K(t^{n+1})}\eta_{i}^{n+1} u_{j,h}^n,\phi_{i,h}^{n+1}\right)&\le\sum\limits_{j=1}^N\frac{\|r_i(t^{n+1})\|_{\infty}}{\inf\limits_\Omega\|K(t^{n+1})\|}\Big|\left(\eta_{i}^{n+1} u_{j,h}^n,\phi_{i,h}^{n+1}\right)\Big|\\&\le \sum\limits_{j=1}^N\frac{\|r_i\|_{\infty,\infty}}{K_{\min}}\|\eta_{i}^{n+1}\|\| u_{j,h}^n\|_\infty\|\|\phi_{i,h}^{n+1}\|\\&\le \frac{C\|r_i\|_{\infty,\infty}}{K_{\min}}\|\eta_{i}^{n+1}\|\|\nabla\phi_{i,h}^{n+1}\|\\&\le \frac{\alpha_i}{10}\|\nabla\phi_{i,h}^{n+1}\|^2+\frac{C\|r_i\|_{\infty,\infty}^2}{\alpha_iK_{\min}^2}\|\eta_{i}^{n+1}\|^2.
\end{align*}
Now we want to find the upper-bound of \begin{align}
    G(t,u_i,\phi_{i,h}^{n+1})=\left(\frac{u_{i}(t^{n+1})-u_{i}(t^{n})}{\Delta  t}- u_{i,t}(t^{n+1}),\phi_{i,h}^{n+1}\right)\nonumber\\+\left(\frac{r_i(t^{n+1}) u_{i}(t^{n+1})}{K(t^{n+1})}\sum\limits_{j=1}^N\big\{u_{j,h}^{n}-u_{j}(t^{n+1})\big\},\phi_{i,h}^{n+1}\right).\label{G-phi}
\end{align}
For some $t^*\in[t^n,t^{n+1}]$, we use Taylor's series expansion, Poincar\'e, Cauchy-Schwarz, and Young's inequalities to obtain the following bound for the first term on the right-hand-side of \eqref{G-phi}
\begin{align*}
    \left(\frac{u_{i}(t^{n+1})-u_{i}(t^{n})}{\Delta  t}- u_{i,t}(t^{n+1}),\phi_{i,h}^{n+1}\right)&=\frac{\Delta t}{2}\left(u_{i,tt}(t^*),\phi_{i,h}^{n+1}\right)\\&\le\frac{C\Delta t}{2}\left(u_{i,tt}(t^*),\nabla\phi_{i,h}^{n+1}\right)\\&\le\frac{C\Delta t}{2}\|u_{i,tt}(t^*)\|\|\nabla\phi_{i,h}^{n+1}\|\\&\le\frac{\alpha_i}{10}\|\nabla\phi_{i,h}^{n+1}\|^2+\frac{C(\Delta t)^2}{\alpha_i}\|u_{i,tt}(t^*)\|^2.
\end{align*}
We can find the upper-bound of the last term on the right-hand-side of \eqref{G-phi} using Taylor's series expansion, Poincar\'e, H\"older's, triangle, and Young's inequalities together with the regularity assumption as
\begin{align*}
    \Bigg(\frac{r_i(t^{n+1}) u_{i}(t^{n+1})}{K(t^{n+1})}&\sum\limits_{j=1}^N\big\{u_{j,h}^{n}-u_{j}(t^{n+1})\big\},\phi_{i,h}^{n+1}\Bigg)\\&\le \frac{C\|r_i(t^{n+1})\|_{\infty}}{\inf\limits_\Omega\|K(t^{n+1})\|}\sum\limits_{j=1}^N\left(|u_{i}(t^{n+1})\big\{u_{j,h}^{n}-u_{j}(t^{n+1})\big\}\phi_{i,h}^{n+1}|\right)\\&\le\frac{C\|r_i\|_{\infty,\infty}}{K_{\min}}\|u_{i}(t^{n+1})\|_\infty\sum\limits_{j=1}^N\|u_{j,h}^{n}-u_{j}(t^{n+1})\|\|\nabla\phi_{i,h}^{n+1}\|\\&\le\sum\limits_{j=1}^N\frac{C\|r_i\|_{\infty,\infty}}{K_{\min}}\|u_{j,h}^{n}-u_{j}(t^{n+1})\|\|\nabla\phi_{i,h}^{n+1}\|\\&\le\sum\limits_{j=1}^N\left(\frac{\alpha_i}{10N}\|\nabla\phi_{i,h}^{n+1}\|^2+\frac{C\|r_i\|_{\infty,\infty}^2}{\alpha_i K_{\min}^2}\|u_{j,h}^{n}-u_{j}(t^{n+1})\|^2\right)\\&\le \frac{\alpha_i}{10}\|\nabla\phi_{i,h}^{n+1}\|^2+\frac{C\|r_i\|_{\infty,\infty}^2}{\alpha_i K_{\min}^2}\sum_{j=1}^N\big(\|u_{j,h}^{n}-u_j(t^n)\|^2+\|u_j(t^n)-u_j(t^{n+1})\|^2\big)\\&\le\frac{\alpha_i}{10}\|\nabla\phi_{i,h}^{n+1}\|^2+\frac{C\|r_i\|_{\infty,\infty}^2}{\alpha_i K_{\min}^2}\sum_{j=1}^N\big(2\|\eta_j^n\|^2+2\|\phi_{j,h}^n\|^2+(\Delta t)^2\|u_{j,t}(s^*)\|^2\big)\\&\le\frac{\alpha_i}{10}\|\nabla\phi_{i,h}^{n+1}\|^2+\frac{C\|r_i\|_{\infty,\infty}^2}{\alpha_i K_{\min}^2}\big(h^{2k+2}+(\Delta t)^2\big)+\frac{C\|r_i\|_{\infty,\infty}^2}{\alpha_i K_{\min}^2}\sum_{j=1}^N\|\phi_{j,h}^n\|^2,
\end{align*}
with $s^*\in[t^n,t^{n+1}]$.
Thus, we have
\begin{align*}
    |G(t,u_i,\phi_{i,h}^{n+1})|\le\frac{\alpha_i}{5}\|\nabla\phi_{i,h}^{n+1}\|^2+C\big(h^{2k+2}+(\Delta t)^2\big)+C\sum_{j=1}^N\|\phi_{j,h}^n\|^2.
\end{align*}
Now, using the above bounds, we can rewrite \eqref{all-phi} as\begin{align}
  \frac{1}{2\Delta t}\left(\|\phi_{i,h}^{n+1}-\phi_{i,h}^{n}\|^2+\|\phi_{i,h}^{n+1}\|^2-\|\phi_{i,h}^{n}\|^2\right)+\frac{\alpha_i}{2}\|\nabla\phi_{i,h}^{n+1}\|^2\le \frac{5d_i^2}{2\alpha_i}\|\nabla \eta_{i}^{n+1}\|^2+C\sum_{j=1}^N\|\phi_{j,h}^n\|^2\nonumber\\+\frac{C(1-\mu_i)^2\|r_i(t^{n+1})\|_\infty^2}{\alpha_i}\|\eta_{i}^{n+1}\|^2+\frac{C\|r_i\|_{\infty,\infty}^2}{\alpha_i K_{\min}^2}\|\eta_{i}^{n+1}\|^2+C\big(h^{2k+2}+(\Delta t)^2\big).
\end{align}
Using the regularity assumption again, we obtain
 \begin{align}
  \frac{1}{2\Delta t}\left(\|\phi_{i,h}^{n+1}-\phi_{i,h}^{n}\|^2+\|\phi_{i,h}^{n+1}\|^2-\|\phi_{i,h}^{n}\|^2\right)+\frac{\alpha_i}{2}\|\nabla\phi_{i,h}^{n+1}\|^2\le C\big(h^{2k}+(\Delta t)^2\big)+C\sum_{j=1}^N\|\phi_{j,h}^n\|^2.
\end{align}
Dropping non-negative term from left-hand-side, multiplying both sides by $2\Delta t$, use $\|\phi_{i,h}^0\|=0$, $\Delta{t}M=T$, and summing over time-steps $n=0,1,\cdots,M-1$ to find
\begin{align}
  \|\phi_{i,h}^{M}\|^2+\alpha_i\Delta t\sum\limits_{n=1}^{M}\|\nabla\phi_{i,h}^{n}\|^2\le \Delta t\sum\limits_{n=1}^{M-1}C\left(\sum_{j=1}^N\|\phi_{j,h}^n\|^2\right)+C\big(h^{2k}+(\Delta t)^2\big).
\end{align}
Sum over $i=1,2,\cdots,N$, we have
\begin{align}
  \sum\limits_{i=1}^N\|\phi_{i,h}^{M}\|^2+\Delta t\sum\limits_{n=1}^{M}\left(\sum\limits_{i=1}^N\alpha_i\|\nabla\phi_{i,h}^{n}\|^2\right)\le \Delta t\sum\limits_{n=1}^{M-1}C\left(\sum_{i=1}^N\|\phi_{i,h}^n\|^2\right)+C\big(h^{2k}+(\Delta t)^2\big).
\end{align}
Applying the discrete Gr\"onwall Lemma \ref{dgl}, we have
\begin{align}
  \sum\limits_{i=1}^N\|\phi_{i,h}^{M}\|^2+\Delta t\sum\limits_{n=1}^{M}\left(\sum\limits_{i=1}^N\alpha_i\|\nabla\phi_{i,h}^{n}\|^2\right)\le C \big(h^{2k}+(\Delta t)^2\big),
\end{align}
which gives
\begin{align}
  \|\phi_{i,h}^{M}\|^2+\alpha_i\Delta t\sum\limits_{n=1}^{M}\|\nabla\phi_{i,h}^{n}\|^2\le C \big(h^{2k}+(\Delta t)^2\big)\hspace{4mm}\text{for}\hspace{1mm}i=1,2,\cdots,N.\label{bound-on-phi}
\end{align}
Use of triangle and Young's inequalities allows us to write
\begin{align}
    \|e_i^M\|^2+\alpha_i\Delta t\sum_{n=1}^M\|\nabla e_i^n\|^2\le 2\left(\|\phi_{i,h}^{M}\|^2+\alpha_i\Delta t\sum\limits_{n=1}^{M}\|\nabla\phi_{i,h}^{n}\|^2+\|\eta_i^M\|^2+\alpha_i\Delta t\sum\limits_{n=1}^{M}\|\nabla\eta_{i}^{n}\|^2\right).
\end{align}
Using regularity assumptions and bound in \eqref{bound-on-phi}, we have
\begin{align}
    \|u_i(T)-u_{i,h}^M\|^2+\alpha_i\Delta t\sum_{n=1}^M\|\nabla \big(u_i(t^n)-u_{i,h}^n\big)\|^2\le C\big(h^{2k}+(\Delta t)^2\big)\hspace{4mm}\text{for}\hspace{1mm}i=1,2,\cdots,N.
\end{align}
Now, summing over $i=1,2,\cdots,N$ completes the proof.
\end{proof}

\begin{theorem} (Error estimate of DBDF-2) For $i=1,2,\cdots,N$, assume $u_i$ solves \eqref{RDE1} and satisfies\begin{align*}
    u_i\in L^\infty\left(0,T;H^{k+1}(\Omega)^d\right),& u_{i,tt}\in L^\infty\left(0,T;L^2(\Omega)^d\right),u_{i,ttt}\in L^\infty\left(0,T;L^2(\Omega)^d\right),\\&r_i\in L^\infty\left(0,T;L^\infty(\Omega)^d\right),\text{ and } K_{\min}>0,
\end{align*}
   if $\alpha_i>0$ then for $\Delta t>0$ the solution $u_{i,h}$ to the Algorithm \ref{Algn2} converges to the true solution with
 \begin{align}
    \|u_i(T)-u_{i,h}^M\|+\Big\{\alpha_i\Delta t\sum\limits_{n=2}^M\|\nabla(u_i(t^n)-u_{i,h}^n)\|^2\Big\}^{1/2}\le C(h^k+\Delta t^2).
\end{align}\label{Convergence-analysis-BDF-2}
\end{theorem}

\begin{proof}
At first we build an error equation at the time level $t^{n+1}$, the continuous variational formulations can be written as $\forall v_{i,h}\in X_h$
\begin{align}
    \Bigg(&\frac{3u_{i}(t^{n+1})-4u_{i}(t^{n})+u_{i}(t^{n-1})}{2\Delta  t},v_{i,h}\Bigg)+d_i\left(\nabla u_{i}(t^{n+1}),\nabla v_{i,h}\right)=(1-\mu_i)\left(r_i(t^{n+1})u_{i}(t^{n+1}),v_{i,h}\right)\nonumber\\&-\left(\frac{r_i(t^{n+1}) u_{i}(t^{n+1})}{K(t^{n+1})}\sum\limits_{j=1}^Nu_{j}(t^{n+1}),v_{i,h}\right)+\left(\frac{3u_{i}(t^{n+1})-4u_{i}(t^{n})+u_{i}(t^{n-1})}{2\Delta  t}- u_{i,t}(t^{n+1}),v_{i,h}\right)\nonumber\\&+\left(f_i(t^{n+1}),v_{i,h}\right).\label{cont-weak-form2}
\end{align}
Subtract \eqref{disc-weak-form2} from \eqref{cont-weak-form2} and then rearranging yields
\begin{align}
    \left(\frac{3e_{i}^{n+1}-4e_{i}^{n}+e_{i}^{n-1}}{2\Delta  t},v_{i,h}\right)+d_i\left(\nabla e_{i}^{n+1},\nabla v_{i,h}\right)-(1-\mu_i)\left(r_i(t^{n+1})e_{i}^{n+1},v_{i,h}\right)\nonumber\\+\sum\limits_{j=1}^N\left(\frac{r_i(t^{n+1})}{K(t^{n+1})}e_i^{n+1}(2u_{j,h}^n-u_{j,h}^{n-1}),v_{i,h}
    \right)+\sum\limits_{j=1}^N\left(\frac{r_i(t^{n+1})}{K(t^{n+1})}u_i(t^{n+1})(2e_{j}^n-e_{j}^{n-1}),v_{i,h}
    \right)\nonumber\\=G(t,u_i,v_{i,h}),\label{error-equation2}
\end{align}
where\begin{align*}
    G(t,u_i,v_{i,h})=\left(\frac{3u_{i}(t^{n+1})-4u_{i}(t^{n})+u_{i}(t^{n-1})}{2\Delta  t}- u_{i,t}(t^{n+1}),v_{i,h}\right)\\-\left(\frac{r_i(t^{n+1}) u_{i}(t^{n+1})}{K(t^{n+1})}\sum\limits_{j=1}^N\big\{u_{j}(t^{n+1})-2u_{j}(t^{n})+u_j(t^{n-1})\big\},v_{i,h}\right).
\end{align*}
Now we decompose the errors as\begin{align*}
    e_{i}^n:& = u_i(t^n)-u_{i,h}^n=(u_i(t^n)-\tilde{u}_i^n)-(u_{i,h}^n-\tilde{u}_i^n):=\eta_{i}^n-\phi_{i,h}^n,
\end{align*}
where $\tilde{u}_i^n: =P_{X_h}^{L^2}(u_i(t^n))\in X_h$ is the $L^2$ projections of $u_j(t^n)$ into $X_h$. Note that $(\eta_{i}^n,v_{i,h})=0\hspace{2mm} \forall v_{i,h}\in X_h$.  Rewriting, we have for $v_{i,h}\in X_h$
\begin{align}
    \Bigg(&\frac{3\phi_{i,h}^{n+1}-4\phi_{i,h}^{n}+\phi_{i,h}^{n-1}}{2\Delta  t},v_{i,h}\Bigg)+d_i\left(\nabla \phi_{i,h}^{n+1},\nabla v_{i,h}\right)-(1-\mu_i)\left(r_i(t^{n+1})\phi_{i,h}^{n+1},v_{i,h}\right)\nonumber\\&+\sum\limits_{j=1}^N\left(\frac{r_i(t^{n+1})}{K(t^{n+1})}\phi_{i,h}^{n+1}(2u_{j,h}^n-u_{j,h}^{n-1}),v_{i,h}
    \right)+\sum\limits_{j=1}^N\left(\frac{r_i(t^{n+1})}{K(t^{n+1})}u_i(t^{n+1})(2\phi_{j,h}^{n}-\phi_{j,h}^{n-1}),v_{i,h}
    \right)\nonumber\\&=d_i\left(\nabla \eta_{i}^{n+1},\nabla v_{i,h}\right)-(1-\mu_i)\left(r_i(t^{n+1})\eta_{i}^{n+1},v_{i,h}\right)\nonumber\\&+\sum\limits_{j=1}^N\left(\frac{r_i(t^{n+1})}{K(t^{n+1})}\eta_i^{n+1}(2u_{j,h}^n-u_{j,h}^{n-1}),v_{i,h}
    \right)+\sum\limits_{j=1}^N\left(\frac{r_i(t^{n+1})}{K(t^{n+1})}u_i(t^{n+1})(2\eta_{j}^{n}-\eta_{j}^{n-1}),v_{i,h}
    \right)\nonumber\\&-G(t,u_i,v_{i,h}).\label{phi-equation2}
\end{align}
Choose $v_{i,h}=\phi_{i,h}^{n+1}$, and use the identity in \eqref{ident}, to obtain
\begin{align}
    &\frac{1}{4\Delta t}\Big(\|\phi_{i,h}^{n+1}\|^2-\|\phi_{i,h}^{n}\|^2+\|2\phi_{i,h}^{n+1}-\phi_{i,h}^{n}\|^2-\|2\phi_{i,h}^{n}-\phi_{i,h}^{n-1}\|^2+\|\phi_{i,h}^{n+1}-2\phi_{i,h}^{n}+\phi_{i,h}^{n-1}\|^2\Big)\nonumber\\&+d_i\|\nabla \phi_{i,h}^{n+1}\|^2-(1-\mu_i)\left(r_i(t^{n+1})\phi_{i,h}^{n+1},\phi_{i,h}^{n+1}\right)+\sum\limits_{j=1}^N\left(\frac{r_i(t^{n+1})}{K(t^{n+1})}\phi_{i,h}^{n+1}(2u_{j,h}^n-u_{j,h}^{n-1}),\phi_{i,h}^{n+1}
    \right)\nonumber\\&+\sum\limits_{j=1}^N\left(\frac{r_i(t^{n+1})}{K(t^{n+1})}u_i(t^{n+1})(2\phi_{j,h}^{n}-\phi_{j,h}^{n-1}),\phi_{i,h}^{n+1}
    \right)=d_i\left(\nabla \eta_{i}^{n+1},\nabla \phi_{i,h}^{n+1}\right)\nonumber\\&-(1-\mu_i)\left(r_i(t^{n+1})\eta_{i}^{n+1},\phi_{i,h}^{n+1}\right)+\sum\limits_{j=1}^N\left(\frac{r_i(t^{n+1})}{K(t^{n+1})}\eta_i^{n+1}(2u_{j,h}^n-u_{j,h}^{n-1}),\phi_{i,h}^{n+1}
    \right)\nonumber\\&+\sum\limits_{j=1}^N\left(\frac{r_i(t^{n+1})}{K(t^{n+1})}u_i(t^{n+1})(2\eta_{j}^{n}-\eta_{j}^{n-1}),\phi_{i,h}^{n+1}
    \right)-G\left(t,u_i,\phi_{i,h}^{n+1}\right).\label{all-phi-2}
\end{align}
Now, we find the upper-bounds of terms in \eqref{all-phi-2}. Using H\"older's, and Poincar\'e inequalities, we have
\begin{align*}
    (1-\mu_i)\left(r_i(t^{n+1})\phi_{i,h}^{n+1},\phi_{i,h}^{n+1}\right)&\le|1-\mu_i|\|r_i(t^{n+1})\|_{\infty}\| \phi_{i,h}^{n+1}\|^2\le C |1-\mu_i|\|r_i\|_{\infty,\infty}\|\nabla \phi_{i,h}^{n+1}\|^2.
\end{align*}
Next, using triangle, H\"older's, and Poincar\'e inequalities together with the Assumption \ref{assumption-1}, we have
\begin{align*}
    -\sum\limits_{j=1}^N\left(\frac{r_i(t^{n+1})}{K(t^{n+1})}\phi_{i,h}^{n+1}(2u_{j,h}^n-u_{j,h}^{n-1}),\phi_{i,h}^{n+1}
    \right)&\le\sum\limits_{j=1}^N\Big\|\frac{r_i(t^{n+1})}{K(t^{n+1})}\Big\|_{\infty}\|2u_{j,h}^n-u_{j,h}^{n-1}\|_{\infty}\|\phi_{i,h}^{n+1}\|^2\\&\le\frac{C\|r_i\|_{\infty,\infty}}{K_{\min}}
    \|\nabla\phi_{i,h}^{n+1}\|^2.
    \end{align*}
    With the assumption $\alpha_i>0$, use H\"older's inequality, Sobolev embedding theorem, Poincar\'e and  Young's inequalities, and regularity assumption, we obtain
    \begin{align*}
      -\sum\limits_{j=1}^N\bigg(\frac{r_i(t^{n+1})}{K(t^{n+1})}u_i(t^{n+1})(2\phi_{j,h}^{n}-\phi_{j,h}^{n-1}),&\phi_{i,h}^{n+1}
    \bigg)\le\sum\limits_{j=1}^N\frac{\|r_i\|_{\infty,\infty}}{K_{\min}}\|u_i(t^{n+1})\|_{L^6}\|\phi_{i,h}^{n+1}\|_{L^3}\|2\phi_{j,h}^{n}-\phi_{j,h}^{n-1}\|\\&\le\sum\limits_{j=1}^N\frac{\|r_i\|_{\infty,\infty}}{K_{\min}}\|u_i(t^{n+1})\|_{H^1}\|\phi_{i,h}^{n+1}\|^\frac{1}{2}\|\nabla\phi_{i,h}^{n+1}\|^\frac{1}{2}\|2\phi_{j,h}^{n}-\phi_{j,h}^{n-1}\|\\&\le\sum\limits_{j=1}^N\left(\frac{\alpha_i}{14N}\|\nabla\phi_{i,h}^{n+1}\|^2+\frac{C\|r_i\|_{\infty,\infty}^2}{\alpha_iK_{\min}^2}\|2\phi_{j,h}^{n}-\phi_{j,h}^{n-1}\|^2\right)\\&\le\frac{\alpha_i}{14}\|\nabla\phi_{i,h}^{n+1}\|^2+\frac{C\|r_i\|_{\infty,\infty}^2}{\alpha_iK_{\min}^2}\sum\limits_{j=1}^N\|2\phi_{j,h}^{n}-\phi_{j,h}^{n-1}\|^2.
    \end{align*}
Use Cauchy-Schwarz, and Young's inequalities, to obtain
\begin{align*}
    d_i\left(\nabla \eta_{i}^{n+1},\nabla \phi_{i,h}^{n+1}\right)\le d_i\|\nabla \eta_{i}^{n+1}\|\|\nabla \phi_{i,h}^{n+1}\|\le\frac{\alpha_i}{14}\|\nabla \phi_{i,h}^{n+1}\|^2+\frac{7d_i^2}{2\alpha_i}\|\nabla \eta_{i}^{n+1}\|^2.
\end{align*}
Using H\"older's, Sobolev embedding theorem,  Poincar\'e, and Young's inequalities, we have
\begin{align*}
    (1-\mu_i)\left(r_i(t^{n+1})\eta_{i}^{n+1},\phi_{i,h}^{n+1}\right)&\le|1-\mu_i|\|r_i(t^{n+1})\|_{\infty}\|\eta_{i}^{n+1}\|\|\nabla\phi_{i,h}^{n+1}\|\\&\le \frac{\alpha_i}{14}\|\nabla\phi_{i,h}^{n+1}\|^2+\frac{7(1-\mu_i)^2\|r_i\|^2_{\infty,\infty}}{2\alpha_i}\|\eta_{i}^{n+1}\|^2.
\end{align*}
Using H\"older's inequality, and triangle inequality, Assumption \ref{assumption-1},  Poincar\'e, and  Young's inequalities, we have
\begin{align*}
    \sum\limits_{j=1}^N\left(\frac{r_i(t^{n+1})}{K(t^{n+1})}\eta_{i}^{n+1} \big(2u_{j,h}^n-u_{j,h}^{n-1}\big),\phi_{i,h}^{n+1}\right)&\le\sum\limits_{j=1}^N\frac{\|r_i\|_{\infty,\infty}}{K_{\min}}\|\eta_{i}^{n+1}\|\|2u_{j,h}^n-u_{j,h}^{n-1}\|_{\infty}\|\phi_{i,h}^{n+1}\|\\&\le \sum\limits_{j=1}^N\frac{\|r_i\|_{\infty,\infty}}{K_{\min}}\|\eta_{i}^{n+1}\|\left(2\| u_{j,h}^n\|_\infty+\| u_{j,h}^{n-1}\|_\infty\right)\|\nabla\phi_{i,h}^{n+1}\|\\&\le \frac{C\|r_i\|_{\infty,\infty}}{K_{\min}}\|\eta_{i}^{n+1}\|\|\nabla\phi_{i,h}^{n+1}\|\\&\le \frac{\alpha_i}{14}\|\nabla\phi_{i,h}^{n+1}\|^2+\frac{C\|r_i\|^2_{\infty,\infty}}{\alpha_iK_{\min}^2}\|\eta_{i}^{n+1}\|^2.
\end{align*}
Using H\"older's inequality, Sobolev embedding theorem,,  Poincar\'e inequality, regularity assumption, and  Young's inequality, we have
\begin{align*}
    \sum\limits_{j=1}^N\left(\frac{r_i(t^{n+1})}{K(t^{n+1})}u_i(t^{n+1})(2\eta_{j}^{n}-\eta_{j}^{n-1}),\phi_{i,h}^{n+1}
    \right)&\le\sum\limits_{j=1}^N\frac{\|r_i\|_{\infty,\infty}}{K_{\min}}\|u_i(t^{n+1})\|_{L^6}\|2\eta_{j}^{n}-\eta_{j}^{n-1}\|\|\phi_{i,h}^{n+1}\|_{L^3}\\&\le\sum\limits_{j=1}^N\frac{\|r_i\|_{\infty,\infty}}{K_{\min}}\|u_i(t^{n+1})\|_{H^1}\|2\eta_{j}^{n}-\eta_{j}^{n-1}\|\|\phi_{i,h}^{n+1}\|^{1/2}\|\nabla\phi_{i,h}^{n+1}\|^{1/2}\\&\le\sum\limits_{j=1}^N\frac{C\|r_i\|_{\infty,\infty}}{K_{\min}}\|u_i\|_{L^\infty\big(0,T;H^1(\Omega)^d\big)}\|2\eta_{j}^{n}-\eta_{j}^{n-1}\|\|\nabla\phi_{i,h}^{n+1}\|\\&\le\frac{\alpha_i}{14}\|\nabla\phi_{i,h}^{n+1}\|^2+\frac{C\|r_i\|_{\infty,\infty}^2}{K_{\min}^2}\sum\limits_{j=1}^N\|2\eta_{j}^{n}-\eta_{j}^{n-1}\|^2.
\end{align*}
Using Taylor’s series, Cauchy-Schwarz and Young’s inequalities the last term is evaluated
as
\begin{align*}
    \Big|-G\left(t,u_i,\phi_{i,h}^{n+1}\right)\Big|&\le \Delta t^2\frac{C\|r_i\|_{\infty,\infty}}{K_{\min}}\|u_i\|_{L^\infty\big(0,T;H^1(\Omega)^d\big)}\sum\limits_{j=1}^N\|u_{j,tt}\|_{L^\infty\big(0,T;L^2(\Omega)^d\big)}\|\nabla\phi_{i,h}^{n+1}\|\\&+C\Delta t^2\|u_{i,ttt}\|_{L^\infty\big(0,T;L^2(\Omega)^d\big)}\|\nabla\phi_{i,h}^{n+1}\|\\&\le \frac{\alpha_i}{7}\|\nabla\phi_{i,h}^{n+1}\|^2+\frac{C\Delta t^4\|r_i\|_{\infty,\infty}^2}{\alpha_iK_{\min}^2}.
\end{align*}
Using these estimates in \eqref{all-phi-2} and reducing yields
\begin{align}
    &\frac{1}{4\Delta t}\Big(\|\phi_{i,h}^{n+1}\|^2-\|\phi_{i,h}^{n}\|^2+\|2\phi_{i,h}^{n+1}-\phi_{i,h}^{n}\|^2-\|2\phi_{i,h}^{n}-\phi_{i,h}^{n-1}\|^2+\|\phi_{i,h}^{n+1}-2\phi_{i,h}^{n}+\phi_{i,h}^{n-1}\|^2\Big)\nonumber\\&+\frac{\alpha_i}{2}\|\nabla\phi_{i,h}^{n+1}\|^2\le \frac{C\|r_i\|_{\infty,\infty}^2}{\alpha_iK_{\min}^2}\sum\limits_{j=1}^N\|2\phi_{j,h}^{n}-\phi_{j,h}^{n-1}\|^2+\frac{7d_i^2}{2\alpha_i}\|\nabla \eta_{i}^{n+1}\|^2+\frac{7(1-\mu_i)^2\|r_i\|^2_{\infty,\infty}}{2\alpha_i}\|\eta_{i}^{n+1}\|^2\nonumber\\&+\frac{C\|r_i\|^2_{\infty,\infty}}{\alpha_iK_{\min}^2}\|\eta_{i}^{n+1}\|^2+\frac{C\|r_i\|_{\infty,\infty}^2}{K_{\min}^2}\sum\limits_{j=1}^N\|2\eta_{j}^{n}-\eta_{j}^{n-1}\|^2+\frac{C\Delta t^4\|r_i\|_{\infty,\infty}^2}{\alpha_iK_{\min}^2}\nonumber\\&\le C\sum\limits_{j=1}^N\|2\phi_{j,h}^{n}-\phi_{j,h}^{n-1}\|^2+Ch^{2k}+Ch^{2k+2}+C\Delta t^4.
\end{align}
Dropping non-negative term from left-hand-side, multiplying both sides by $4\Delta t$, using $\|\phi_{i,h}^0\|=\|\phi_{i,h}^1\|=0$, and summing over the time-steps $n=1,2,\cdots,M-1$, we have
\begin{align}
    \|\phi_{i,h}^{M}\|^2+\|2\phi_{i,h}^{M}-\phi_{i,h}^{M-1}\|^2+2\alpha_i\Delta t\sum\limits_{n=2}^M\|\nabla\phi_{i,h}^n\|^2\nonumber\\\le C\Delta t\sum\limits_{n=1}^{M-1}\sum\limits_{j=1}^N\|2\phi_{j,h}^{n}-\phi_{j,h}^{n-1}\|^2+C(h^{2k}+\Delta t^4).
\end{align}
Sum over $i=1,2,\cdots,N$, drop non-negative terms from left-hand-side, and reducing, gives
\begin{align*}
    \sum\limits_{i=1}^N\|\phi_{i,h}^{M}\|^2+2\alpha_i\Delta t\sum\limits_{n=2}^M\sum\limits_{i=1}^N\|\nabla\phi_{i,h}^n\|^2\le C\Delta t\sum\limits_{n=2}^{M-1}\sum\limits_{i=1}^N\|\phi_{i,h}^n\|^2+C(h^{2k}+\Delta t^4).
\end{align*}
Applying the discrete Gr\"onwall Lemma \ref{dgl}, we have
\begin{align}
    \sum\limits_{i=1}^N\|\phi_{i,h}^{M}\|^2+2\alpha_i\Delta t\sum\limits_{n=2}^M\sum\limits_{i=1}^N\|\nabla\phi_{i,h}^n\|^2\le C(h^{2k}+\Delta t^4),
\end{align}
for $i=1,2,\cdots,N$, which gives
\begin{align}
    \|\phi_{i,h}^{M}\|^2+2\alpha_i\Delta t\sum\limits_{n=2}^M\|\nabla\phi_{i,h}^n\|^2\le C(h^{2k}+\Delta t^4).
\end{align}
Use of triangle and Young's inequalities, and regularity assumption completes the proof.
\end{proof}

Now we proof the assumption  $\|u_{i,h}^n\|_{\infty}\le C$ that was used in stability Theorems \ref{stability-theorm}-\ref{stability-theorm2}  and in convergence Theorems \ref{Convergence-analysis-BE}-\ref{Convergence-analysis-BDF-2} by principle of mathematical induction. The strategy of this proof is adopted from the idea of Wong in the analysis of three-species competition model \cite{wong2009analysis}.

\begin{lemma}
 $\|u_{i,h}^n\|_{\infty}\le C$, $\forall n\in\mathbb{N}$.\label{lemma-discrete-bound}
\end{lemma}

\begin{proof}
Basic step: $u_{i,h}^0=I_h(u_i(0,\bx)),$
where $I_h$ is an appropriate interpolation operator. If $u_i(0,\bx)$ is sufficiently regular for $\bx\in\Omega$, we have $\|u_{i,h}^0\|_\infty\le C$, for some constant $C>0$.

Inductive step: Assume for some $L\in\mathbb{N}$, $\|u_{i,h}^L\|_\infty\le C$ holds true. Then we have
\begin{align*}
    \|u_{i,h}^{L+1}\|_\infty&\le Ch^{-\frac{3}{2}}\|u_{i,h}^{L+1}\|\hspace{2mm}\text{(Agmon’s inequality \cite{Robinson2016Three-Dimensional} and discrete inverse inequality)}\\&=Ch^{-\frac{3}{2}}\|u_{i,h}^{L+1}-u_i(t^{L+1})+u_i(t^{L+1})\|\\&\le Ch^{-\frac{3}{2}}\left(\|u_{i,h}^{L+1}-u_i(t^{L+1})\|+\|u_i(t^{L+1})\|\right)\hspace{2mm}\text{(Triangle inequality)}\\&\le Ch^{-\frac{3}{2}}\left(\|\phi_{i,h}^{L+1}\|+\|\eta_i^{L+1}\|+\|u_i(t^{L+1})\|\right)\hspace{2mm}\text{(Triangle inequality)}.
\end{align*}
Using inductive hypothesis in \eqref{bound-on-phi}, and regularity assumption, the above bound can be written as
\begin{align}
    \|u_{i,h}^{L+1}\|_\infty\le C \big(h^{k-\frac{3}{2}}+h^{-\frac{3}{2}}\Delta t+h^{k-\frac{1}{2}}+h^{-\frac{3}{2}}\big).
\end{align}
Therefore, $\|u_{i,h}^{L+1}\|_\infty\le C$ holds also true. Hence, by the principle of mathematical induction, $\|u_{i,h}^n\|_\infty\le C$ holds true  $\forall n\in\mathbb{N}\cup\{0\}$.
\end{proof}

\section{Numerical tests}\label{numerical-experiments}
In this section, we perform several numerical experiments to support theoretical results and to explain the harvesting or stocking effect on population density from the simulated outcomes. In all the experiments, we consider a domain $\Omega=(0,1)\times (0,1)$; we also use $\mathbb{P}_2$ element for the finite element computation, and structured triangular meshes. We define the average energy density corresponding to a species density $u_i$ at time $t=t^n$ as $$\Bar{u}_i^n=\frac{1}{|\Omega|}\int_{\tau_h} u_{i}(t^n,\bx)d\bx.$$
The experiment that involves the second-order accurate DBDF-2 scheme uses the first-order accurate DBE scheme at the first time-step to generate the required number of initial conditions.

In the first experiment, we numerically verified the predicted convergence rates. We observed the evolution of population density with an exponentially varying carrying capacity in the second experiment. In the third experiment, we observed the effect of diffusion parameters on population density. The effect of harvesting and stocking on the evolution of species density is presented in the fourth and fifth experiments, respectively. The numerical experiments were done in the finite element platform Freefem++ \cite{MR3043640} using the direct solver UMFPACK \cite{davis2004algorithm}.

\begin{table}[h]
    \begin{center}
    \begin{tabular}{|l|p{3.4cm}|p{1.6cm}|p{1.3cm}|p{5.1cm}| }\hline
\textbf{Test}      
& \textbf{Description}   
& \textbf{Carrying Capacity}& \textbf{Growth Rate}& \textbf{Additional Parameters}\\\hline
\multirow{4}{*}{1}
& \multirow{4}{*}{Verify the convergence}\vspace{-1.5ex}
\multirow{5}{*}{rates}
& \multirow{4}{*}{Periodic}& \multirow{4}{*}{Periodic} & $N=2$, $\mu_1=0.001$, $\mu_2=0.0006$, $d_1=d_2=1$\\\cline{5-5}
& & & &$N=3$, $\mu_1=0.001$, $\mu_2=0.0006$, $\mu_3=0.0$, $d_1=d_2=d_3=1$\\\hline
\multicolumn{5}{|c|}{\textbf{Effect on population density of varying}}\\\hline
2      
& spatio-temporal carrying capacity                         
& Gaussian-periodic & Constant&$N=3$, $\mu_1=0.0009$, $\mu_2=0.0015$, $\mu_3=0.0027$, $d_1=d_2=d_3=1$  \\\hline
3       
& diffusion speed 
&Gaussian-periodic& Periodic &$N=3$, $\mu_1=\mu_2=\mu_3=0.0$ \\\hline
4 &
 harvesting parameters
& Gaussian-periodic&Periodic &\multirow{2}{*}{$N=3$, $d_1=0.1$, $d_2=0.02$,}\vspace{-1.3ex}
\multirow{3}{*}{$d_3=0.01$} \\\cline{1-4}
5 & stocking parameters & Gaussian-periodic& Periodic&  \\\hline
    \end{tabular}
    \end{center}
    \caption{A brief summary of the numerical experiments where $\mu_i$ is the harvesting coefficient and $d_i$ is the diffusion speed of the $i^{th}$ competing species.}
\end{table}

\subsection{Test 1: Convergence rate verification} 
We define the global error $e_i:=u_i-u_{i,h}$ and its $L^2$-$H^1$ norm as $\|e_i\|_{2,1}:=\|e_i\|_{L^2\big(0,T;H^1(\Omega)^d\big)}$. We have seen from the convergence analysis that the predicted error of the Algorithm \ref{Algn1} and \ref{Algn2} and  for $\mathbb{P}_2$ finite element are
\begin{align}
    \|u_i-u_{i,h}\|_{2,1}&\le C(h^2+\Delta t),\hspace{2mm}\text{and}\label{error-estimate-BE}\\
    \|u_i-u_{i,h}\|_{2,1}&\le C(h^2+\Delta t^2),\hspace{1mm}i=1,2,\cdots,N,\label{error-estimate-BDF-2}
\end{align}
respectively. To verify the above convergence rates, we plugin the following carrying capacity and intrinsic growth rates
\begin{align*}
    K(t,\bx)=((2.1+\cos(x)\cos(y))(1.1+\cos(t)),\hspace{1mm}\text{and}\hspace{1mm}
    r_i(t,\bx)=(1.5+\sin(x)\sin(y))(1.2+\sin(t)),
\end{align*} respectively, in 
\begin{align}
    f_i(t,\bx)=\frac{\partial u_i}{\partial t}-d_i\Delta u_i-r_iu_i\left(1-\mu_i-\frac{1}{K}\sum\limits_{j=1}^N u_j\right),\label{forcing-function}
\end{align}
to obtain the forcing $f_i(t,\bx)$, for $i=1,2,\cdots,N$. For this experiment, we consider known analytical solution as the Dirichlet boundary condition on the boundary of the unit square, the diffusion coefficients are $d_i=1$, for $i=1,2,\cdots, N$. To observe the spatial convergence rates, we keep fixed, a small simulation end time $T$, successively reduce mesh size $h$ and run the simulations, and record the errors. On the other hand, to exhibit the temporal convergence, we use a fixed small mesh size $h$, successively refined time-step size $\Delta t$ and run the simulation, and record the errors.

\subsubsection{Two-species competition model}
In this case, we have $N=2$, and consider the  following analytical solution
\begin{align*}
    u_1(t,\bx) &= \big(1.1 + \sin(t)\big)(2.0 + \sin(y)),\\
    u_2(t,\bx) &=  \big(2.0 + \cos(t)\big)\big(1.1+\cos(x)\big),
\end{align*}
together with the harvesting coefficients $\mu_1=0.001$, and $\mu_2=0.0006$. To compute the spatial errors and convergence rates, we consider end time $T=0.0001$, and time-step size $\Delta t=T/8$. The spatial errors and convergence rates for both the DBE and DBDF-2 schemes are given in Table \ref{spatial-convergence-N-2}. We observe second order spatial convergence in both species from both the algorithms, which are consistent with \eqref{error-estimate-BE} and \eqref{error-estimate-BDF-2}, since we have used $\mathbb{P}_2 $ element. For the temporal convergence rate, we keep fixed $T=1$, and $h=1/64$, and present the temporal errors and convergence rates in Table \ref{temporal-convergence-N-2}. It is observed the first order temporal convergence rate from the DBE scheme, which is optimal rate for the backward-Euler time-stepping algorithm, and is an excellent agreement with the error estimate in \eqref{error-estimate-BE}. Recall that we used the backward-Euler formula to approximate the time derivative. Whereas, we observe a second order temporal convergence rate for the DBDF-2 scheme, which is also optimal as we approximate the temporal derivative by the BDF-2 formula, and is consistent with \eqref{error-estimate-BDF-2}.

\begin{center}
\begin{table}[!ht]
	\begin{center}
		\small\begin{tabular}{|c|c|c|c|c|c|c|c|c|}\hline
			\multicolumn{9}{|c|}{Errors and convergence rates (fixed $T=0.0001$, $\Delta t=T/8$)}\\\hline
			&\multicolumn{4}{c|}{DBE scheme}&\multicolumn{4}{c|}{DBDF-2 scheme}\\\hline
			$h$ & $\|e_1\|_{2,1}$ & rate   &$\|e_2\|_{2,1}$ & rate  &$\|e_1\|_{2,1}$ & rate   &$\|e_2\|_{2,1}$ & rate\\ \hline
			$1/4$ & 2.1868e-05 &  & 3.6307e-05 &  &2.0456e-05 &&   3.3962e-05 &\\\hline
			$1/8$ & 5.4640e-06 &2.00& 9.1106e-06 & 1.99&5.1111e-06& 2.00&    8.5221e-06&1.99\\\hline
			$1/16$ & 1.3658e-06 & 2.00 & 2.2798e-06 &2.00&1.2776e-06  &2.00&  2.1325e-06&2.00\\\hline
			$1/32$ & 3.4144e-07 & 2.00 & 5.7009e-07 & 2.00&3.1939e-07&2.00&    5.3327e-07&2.00\\\hline
			$1/64$ & 8.5360e-08& 2.00 & 1.4253e-07& 2.00&7.9848e-08 &2.00&   1.3333e-07&2.00\\\hline
		\end{tabular}
	\end{center}
	\caption{\footnotesize Two-species model: Spatial errors and convergence rates with $\mu_1=0.001$, and $\mu_2=0.0006$.} \label{spatial-convergence-N-2}
\end{table}
\end{center}

\begin{table}[!ht]
	\begin{center}
		\small\begin{tabular}{|c|c|c|c|c|c|c|c|c|}\hline
			\multicolumn{9}{|c|}{Errors and convergence rates (fixed $T=1$, $h = 1/64$)}\\\hline
			&\multicolumn{4}{c|}{DBE scheme}&\multicolumn{4}{c|}{DBDF-2 scheme}\\\hline
			$\Delta t$ & $\|e_1\|_{2,1}$ & rate   &$\|e_2\|_{2,1}$ & rate  &$\|e_1\|_{2,1}$ & rate   &$\|e_2\|_{2,1}$ & rate\\ \hline
			$\frac{T}{4}$ & 4.9154e-02 &  & 6.5044e-02 & & 2.3061e-01 &&   3.3998e-01    &\\ \hline
			$\frac{T}{8}$ & 2.3436e-02 & 1.07 & 3.2022e-02 & 1.02&  5.1569e-02  &2.16&  8.3617e-02 &2.02\\\hline
			$\frac{T}{16}$ & 1.1468e-02 & 1.03 & 1.5931e-02 & 1.01& 1.2773e-02  &2.01&  2.1564e-02  &1.96\\\hline
			$\frac{T}{32}$ & 5.6853e-03 & 1.01 & 7.9691e-03 & 1.00 & 3.2627e-03 &1.97&   5.6131e-03 & 1.94\\\hline
			$\frac{T}{64}$ & 2.8339e-03 & 1.00 & 3.9917e-03 & 1.00 &8.3346e-04  &1.97&  1.4477e-03 & 1.96\\\hline
			$\frac{T}{128}$ & 1.4155e-03 & 1.00 & 1.9989e-03 & 1.00  &2.1203e-04 &1.97& 3.6976e-04 &1.97\\\hline
		\end{tabular}
	\end{center}
	\caption{\footnotesize Two-species model: Temporal errors and convergence rates with $\mu_1=0.001$, and $\mu_2=0.0006$.}\label{temporal-convergence-N-2}
\end{table}

\subsubsection{Three-species competition model}
In this case, we have $N=3$, and consider the following manufactured analytical solution
\begin{align*}
    u_1(t,\bx) &= \big(1.1 + \sin(t)\big)(2.0 + \sin(y)),\\
    u_2(t,\bx) &=  \big(2.0 + \cos(t)\big)\big(1.1+\cos(x)\big),\\
    u_3(t,\bx)&=\big(1.1+\sin(t)\big)\big(1.1+\cos(y)\big),
\end{align*}
together with the harvesting coefficients $\mu_1=0.001$, $\mu_2=0.0006$, and $\mu_3=0.0$.
We then compute the solution using the both Algorithm \ref{Algn1} and \ref{Algn2} and compare them with the manufactured analytical solution.  

\begin{center}
\begin{table}[!ht]
	\begin{center}
		\small\begin{tabular}{|c|c|c|c|c|c|c|}\hline
			\multicolumn{2}{|c}{DBE scheme}&\multicolumn{5}{|c|}{Errors and convergence rates (fixed $T=0.0001$, $\Delta t=T/8$)}\\\hline
			$h$ & $\|e_1\|_{2,1}$ & rate   &$\|e_2\|_{2,1}$ & rate &$\|e_3\|_{2,1}$ & rate \\ \hline
			$\frac{1}{4}$ & 2.1868e-05 &  & 3.6307e-05 & & 1.3313e-05&  \\\hline
			$\frac{1}{8}$ & 5.4640e-06 &2.00& 9.1106e-06 & 1.99& 3.3407e-06& 1.99\\\hline
			$\frac{1}{16}$ & 1.3658e-06 & 2.00 & 2.2798e-06 &2.00& 8.3597e-07&2.00\\\hline
			$\frac{1}{32}$ & 3.4144e-07 & 2.00 & 5.7009e-07 & 2.00& 2.0904e-07&2.00\\\hline
			$\frac{1}{64}$ & 8.5364e-08& 2.00 & 1.4254e-07& 2.00& 5.2268e-08& 2.00\\\hline
		\end{tabular}
	\end{center}
	\caption{\footnotesize Three-species model: Spatial errors and convergence rates with $\mu_1=0.001$, $\mu_2=0.0006$, and $\mu_3=0.0$.} \label{spatial-convergence-N-3}
\end{table}
\end{center}

\begin{center}
\begin{table}[!ht]
	\begin{center}
		\small\begin{tabular}{|c|c|c|c|c|c|c|}\hline
			\multicolumn{2}{|c}{DBDF-2 scheme}&\multicolumn{5}{|c|}{Errors and convergence rates (fixed $T=0.001$, $\Delta t=T/16$)}\\\hline
			$h$ & $\|e_1\|_{2,1}$ & rate   &$\|e_2\|_{2,1}$ & rate &$\|e_3\|_{2,1}$ & rate \\ \hline
			$\frac{1}{4}$ &   6.6986e-05 &&   1.1115e-04 && 4.0774e-05&\\\hline
			$\frac{1}{8}$ & 1.6738e-05 &2.00&   2.7893e-05 &1.99& 1.0233e-05&1.99\\\hline
			$\frac{1}{16}$ &4.1839e-06 & 2.00&   6.9804e-06 &2.00& 2.5608e-06&2.00 \\\hline
			$\frac{1}{32}$ & 1.0460e-06 &2.00&   1.7457e-06 &2.00& 6.4041e-07& 2.00\\\hline
			$\frac{1}{64}$ &2.6174e-07  &2.00&  4.3698e-07 &2.00& 1.6034e-07 &2.00\\\hline
		\end{tabular}
	\end{center}
	\caption{\footnotesize Three-species model: Spatial errors and convergence rates with $\mu_1=0.001$, $\mu_2=0.0006$, and $\mu_3=0.0$.} \label{spatial-convergence-N-3-BDF2}
\end{table}
\end{center}

\begin{table}[!ht]
	\begin{center}
		\small\begin{tabular}{|c|c|c|c|c|c|c|}\hline
			\multicolumn{2}{|c}{DBE scheme}&\multicolumn{5}{|c|}{Temporal convergence (fixed $h = 1/64$)}\\\hline
			$\Delta t$ & $\|e_1\|_{2,1}$ & rate   &$\|e_2\|_{2,1}$ & rate &$\|e_3\|_{2,1}$ & rate  \\ \hline
			$\frac{T}{4}$ & 2.1742e-01   &      &  3.1934e-01 & &  1.6932e-01  & \\ \hline
			$\frac{T}{8}$ & 9.9975e-02  & 1.12 &  1.5053e-01 & 1.09 &  7.7848e-02 & 1.12 \\\hline
			$\frac{T}{16}$ & 4.8267e-02 & 1.05 &  7.3538e-02 & 1.03 &  3.7585e-02 & 1.05\\\hline
			$\frac{T}{32}$ & 2.3770e-02 & 1.02 &  3.6446e-02 & 1.01 &  1.8510e-02 &1.02 \\\hline
			$\frac{T}{64}$ & 1.1805e-02 & 1.01 &  1.8163e-02& 1.00&   9.1929e-03 &1.01\\\hline
			$\frac{T}{128}$ & 5.8844e-03& 1.00 & 9.0705e-03& 1.00 &   4.5823e-03 & 1.00\\\hline
		\end{tabular}
	\end{center}
	\caption{\footnotesize Three-species model: Temporal errors and convergence rates with $\mu_1=0.001$, $\mu_2=0.0006$, and $\mu_3=0.0$.}\label{temporal-convergence-N3-BE}
\end{table}

\begin{table}[!ht]
	\begin{center}
		\small\begin{tabular}{|c|c|c|c|c|c|c|}\hline
			\multicolumn{2}{|c}{DBDF-2 scheme}&\multicolumn{5}{|c|}{Temporal convergence (fixed $h = 1/128$)}\\\hline
			$\Delta t$ & $\|e_1\|_{2,1}$ & rate   &$\|e_2\|_{2,1}$ & rate &$\|e_3\|_{2,1}$ & rate  \\ \hline
			$\frac{T}{4}$ &  4.2702e-01  &&  6.4799e-01 && 3.3294e-01 &\\ \hline
			$\frac{T}{8}$ & 8.8706e-02 & 2.27 & 1.5471e-01 & 2.07 & 6.8771e-02& 2.28\\\hline
			$\frac{T}{16}$ & 2.1412e-02 & 2.05&   3.9362e-02 & 1.97& 1.6557e-01 &2.05\\\hline
			$\frac{T}{32}$ & 5.4327e-03 & 1.98& 1.0206e-02 & 1.95& 4.1909e-03 & 1.98\\\hline
			$\frac{T}{64}$ & 1.3844e-03  & 1.97&  2.6272e-03 & 1.96&  1.0655e-03 & 1.98\\\hline
			$\frac{T}{128}$ & 3.5123e-04 &1.98& 6.6970e-04 &1.97& 2.6976e-04 & 1.98\\\hline
		\end{tabular}
	\end{center}
	\caption{\footnotesize Three-species model: Temporal errors and convergence rates with $\mu_1=0.001$, $\mu_2=0.0006$, and $\mu_3=0.0$.}\label{temporal-convergence-N3-BDF2}
\end{table}

The spatial errors and convergence rates for the Algorithm \ref{Algn1}, and Algorithm \ref{Algn2} are given in Table \ref{spatial-convergence-N-3} in Table \ref{spatial-convergence-N-3-BDF2}, respectively. We observe second order spatial convergence in both algorithms for all species, which is consistent with both \eqref{error-estimate-BE} and \eqref{error-estimate-BDF-2}. The temporal errors and convergence rates for the Algorithm \ref{Algn1}, and Algorithm \ref{Algn2} are presented in Table \ref{temporal-convergence-N3-BE} and in Table \ref{temporal-convergence-N3-BDF2}, respectively. From Table \ref{temporal-convergence-N3-BE} we see, the DBE scheme exhibits first order temporal convergence rate, which is optimal rate as a backward-Euler time-stepping algorithm, and is an excellent agreement with the error estimate in \eqref{error-estimate-BE}. 

On the other hand, DBDF-2 scheme displays second order temporal convergence for all three species in Table \ref{temporal-convergence-N3-BDF2}, which is also optimal for the second order time-stepping algorithm and is consistent with the error estimate in \eqref{error-estimate-BDF-2}. Therefore, for both two and three species, we observe optimal convergence rates.

In all the experiments henceforward, we consider the carrying capacity $$K(t,\bx)\equiv\big(1.2+2.5\pi^2e^{-(x-0.5)^2-(y-0.5)^2}\big)\big(1.0+0.3\cos(t)\big),$$ initial population density $u_i^0=1.6$, forcing functions $f_i\equiv 0$ for $i=1,2,3$, and solve \eqref{RDE1} using the DBDF-2 scheme given in Algorithm \ref{Algn2} with time-step size $\Delta t=0.1$ along with no-flux boundary condition. The no-flux boundary condition ensures the competing species live in a closed environment.

\subsection{Test 2: Evolution of population density with exponentially varying carrying capacity}\label{exp-2}

For this experiment, we consider constant intrinsic growth rates $r_i\equiv 1$, diffusion rates $d_i=1$, for $i=1,2,3$, and the harvesting coefficients $\mu_1=0.0009$, $\mu_2=0.0015$, and $\mu_3=0.0027$. 
\begin{figure} [ht]
		\centering
		
		\subfloat[]{\includegraphics[width=0.5\textwidth,height=0.3\textwidth]{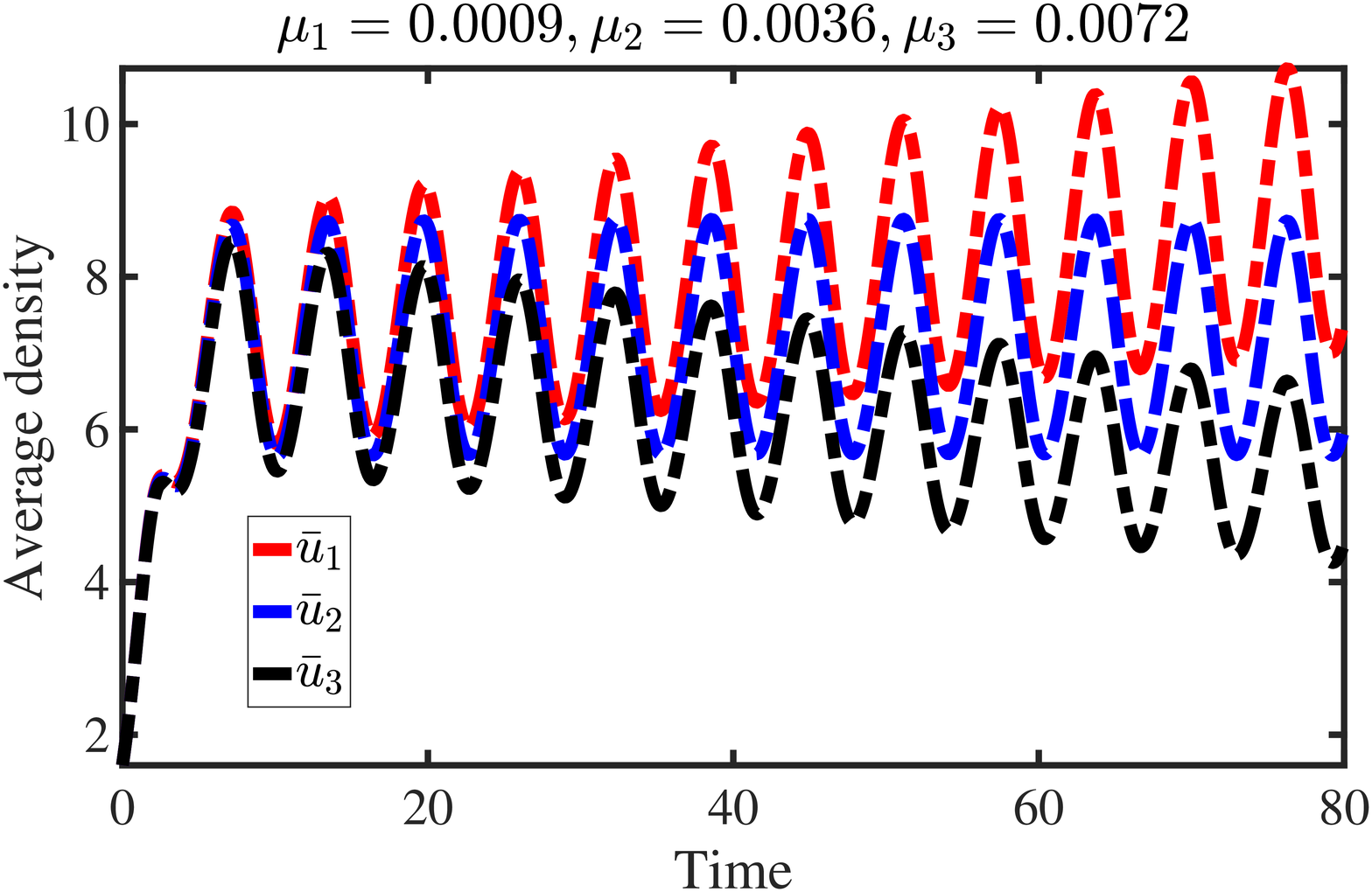}}
		\subfloat[]{\includegraphics[width=0.5\textwidth,height=0.3\textwidth]{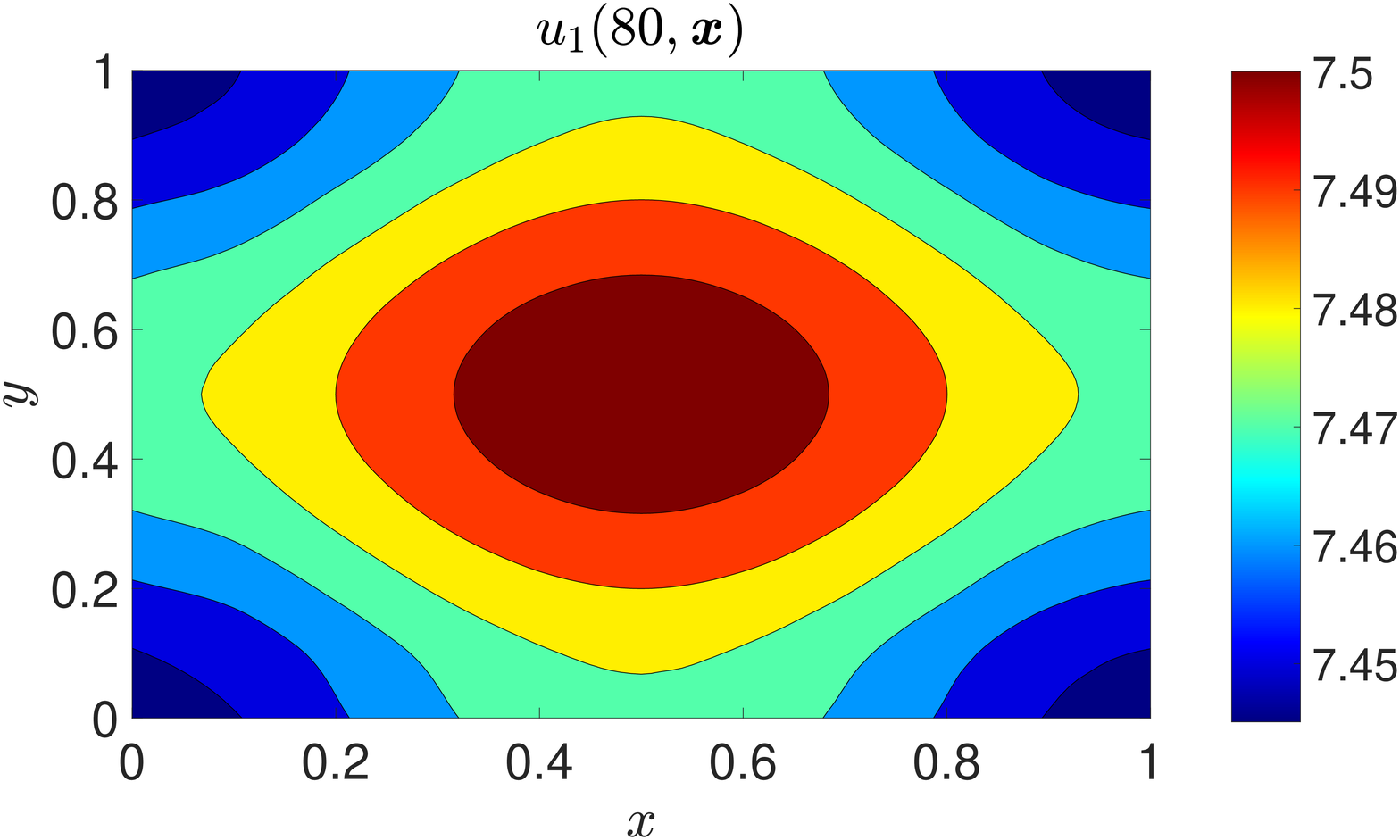}}\\
		
		\subfloat[]{\includegraphics[width=0.5\textwidth,height=0.3\textwidth]{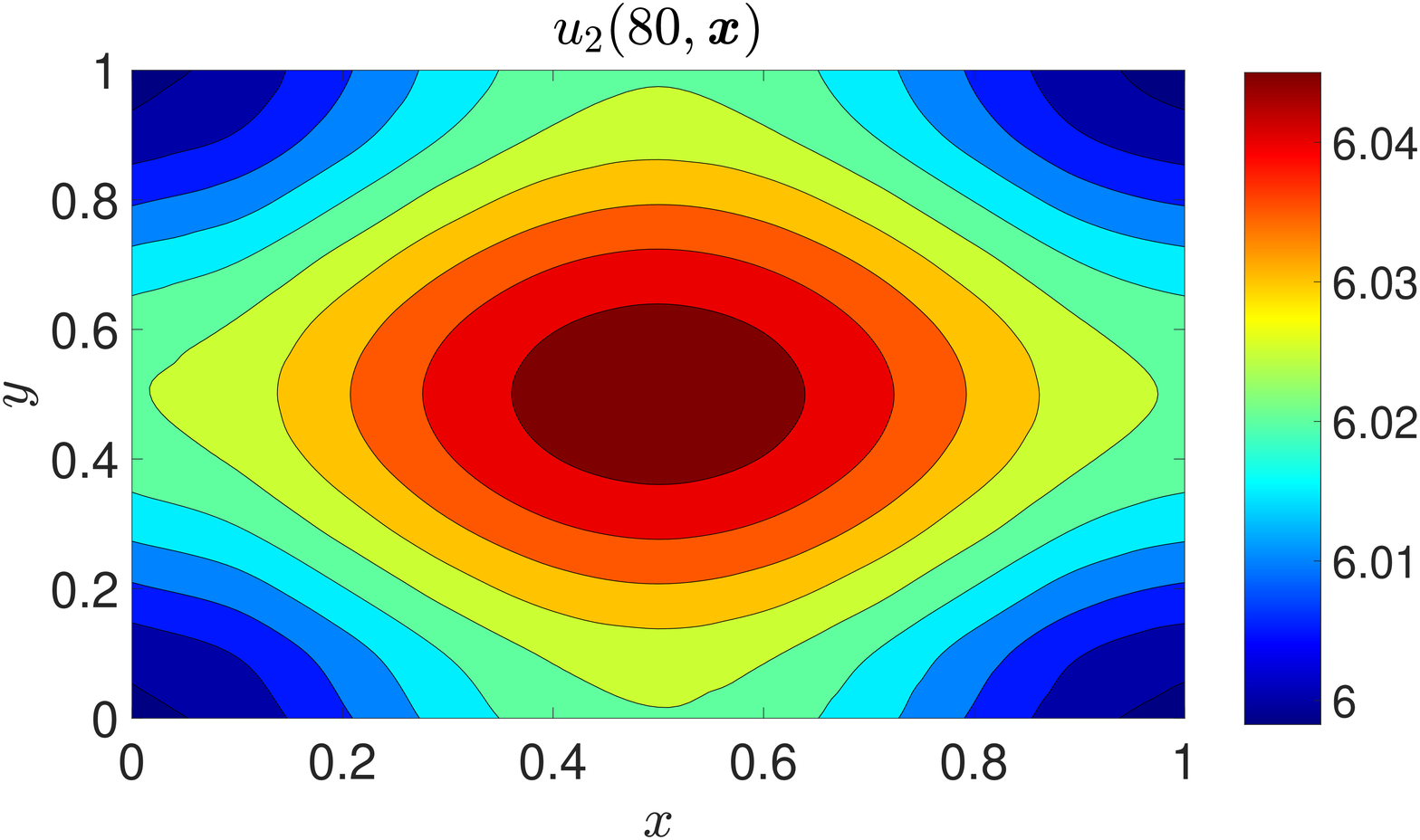}}
		\subfloat[]{\includegraphics[width=0.5\textwidth,height=0.3\textwidth]{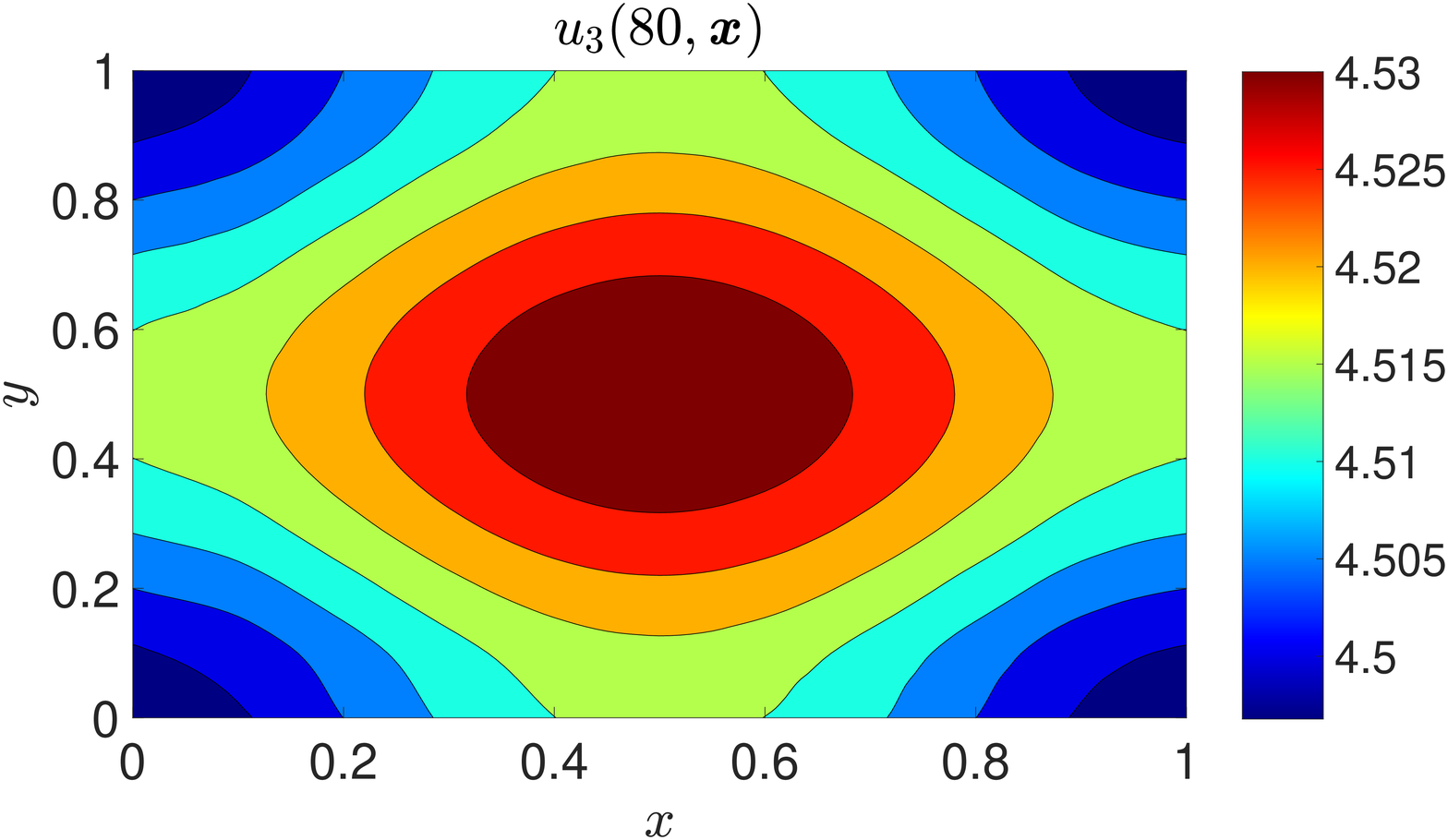}}
		\caption{ (a) Average density of each species,   (b) contour plot of species density $u_1$, (c) contour plot of species density $u_2$, and (d) contour plot of species density $u_3$ with the harvesting coefficients $\mu_1=0.0009$, $\mu_2=0.0015$, and $\mu_3=0.0027$  at $t=80$.}	\label{RDE-energy-0-0009-nu-0-0025-K-time-exp}
	\end{figure}
When carrying capacity $K$ is time periodic, as it is realistic to assume where there is seasonal variation, we display the space averaged profile as a function of time to show its approach to a periodic state, and display the instantaneous contour plot of $u_1,\;u_2$, and $u_3$ for $t = T$, where $ T $ is chosen large enough for time periodicity of $u_i,\;(i=1,2,3)$ to emerge.

In Figures \ref{RDE-energy-0-0009-nu-0-0025-K-time-exp}, the average density of each species versus time is plotted for time $t=0$ to $80$, and the population density contour plot of each of the species at time $t=80$. From the average density plot, we observe periodic population densities for all species, where  the density of  $u_3$ is decreasing because of its higher harvesting coefficient (Figure \ref{RDE-energy-0-0009-nu-0-0025-K-time-exp}(a)). It is predicted that the species $u_3$ will die out if time is too large, and we consider the extinction scenario in a later experiment.

From the contour plots, it is observed that the highest population density is at the point $(0.5,0.5)$ and there is a coexistence of all species, though the population density of the species $u_1$ remains bigger than the species $u_2$, and $u_3$ over the domain (Figure \ref{RDE-energy-0-0009-nu-0-0025-K-time-exp}(b)-(d)). This happens because of different harvesting parameters, and the optimal value of the carrying capacity function is achieved at the point $(0.5,0.5)$, which shows the symmetric distribution of the population.

In all the experiments henceforward unless otherwise stated, we will use the intrinsic growth rates $r_i(t,\bx)\equiv(1.5+\sin(x)\sin(y))(1.2+\sin(t)),\hspace{1mm}i=1,2,3.$

\subsection{Test 3: Diffusion speed and evolution of population density}

In this example, we consider the problem for three species populations in absence of harvesting (e.g., $\mu_i=0$). We want to see how the average population density of a species varies with the diffusion parameter. We plot the average density versus time in Figure \ref{diffusion_varies} for all three species varying the diffusion parameters as $d_i=0.01,\; 0.02$, and $0.1$ for $i=1,2,3$. Figure \ref{diffusion_varies} suggests that the initial value is unimportant to the final state due to the global convergence of solutions. It is also remarked that the same is true in other experiments.
From all three plots in Figure \ref{diffusion_varies} (a)-(c), we observe that as the diffusion parameter increases, the species density increases over time. The species with higher diffusion rate will converge to the stable solution faster \cite{braverman2015competitive}. Figure \ref{diffusion_varies} (c)-(d) are plotted for the same data, but for short and long time scenarios. We observe that the species with the highest diffusion rate is extinct, whereas the species with the lowest diffusion rate is the winner over the other species. In summary, the slow diffuser is the sole winner for multiple population competition and is independent of any choice of equal intrinsic growth rate and the initial population size.

\begin{figure} [ht]
		\centering
		\subfloat[]{\includegraphics[width=0.5\textwidth,height=0.3\textwidth]{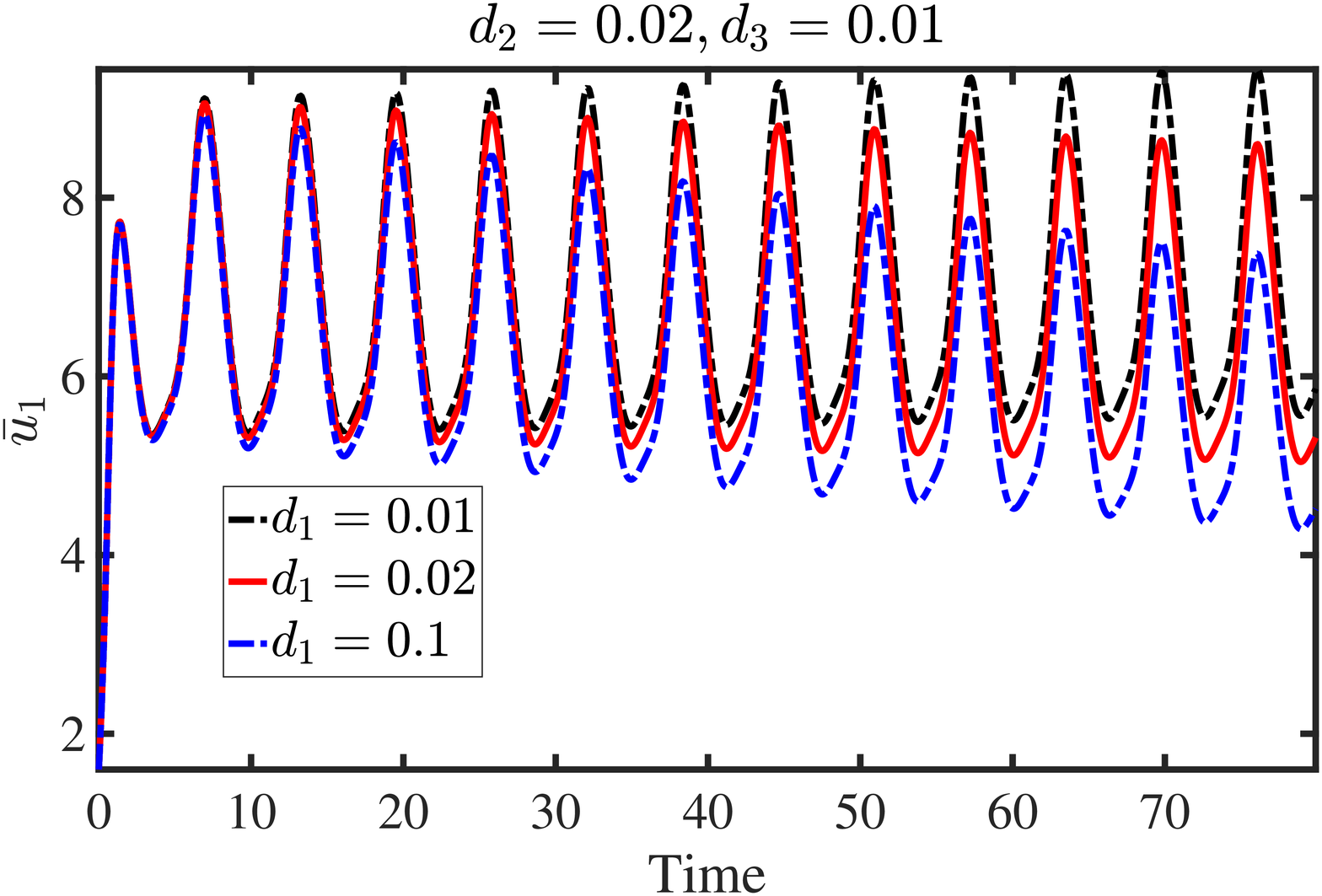}}
		\subfloat[]{\includegraphics[width=0.5\textwidth,height=0.3\textwidth]{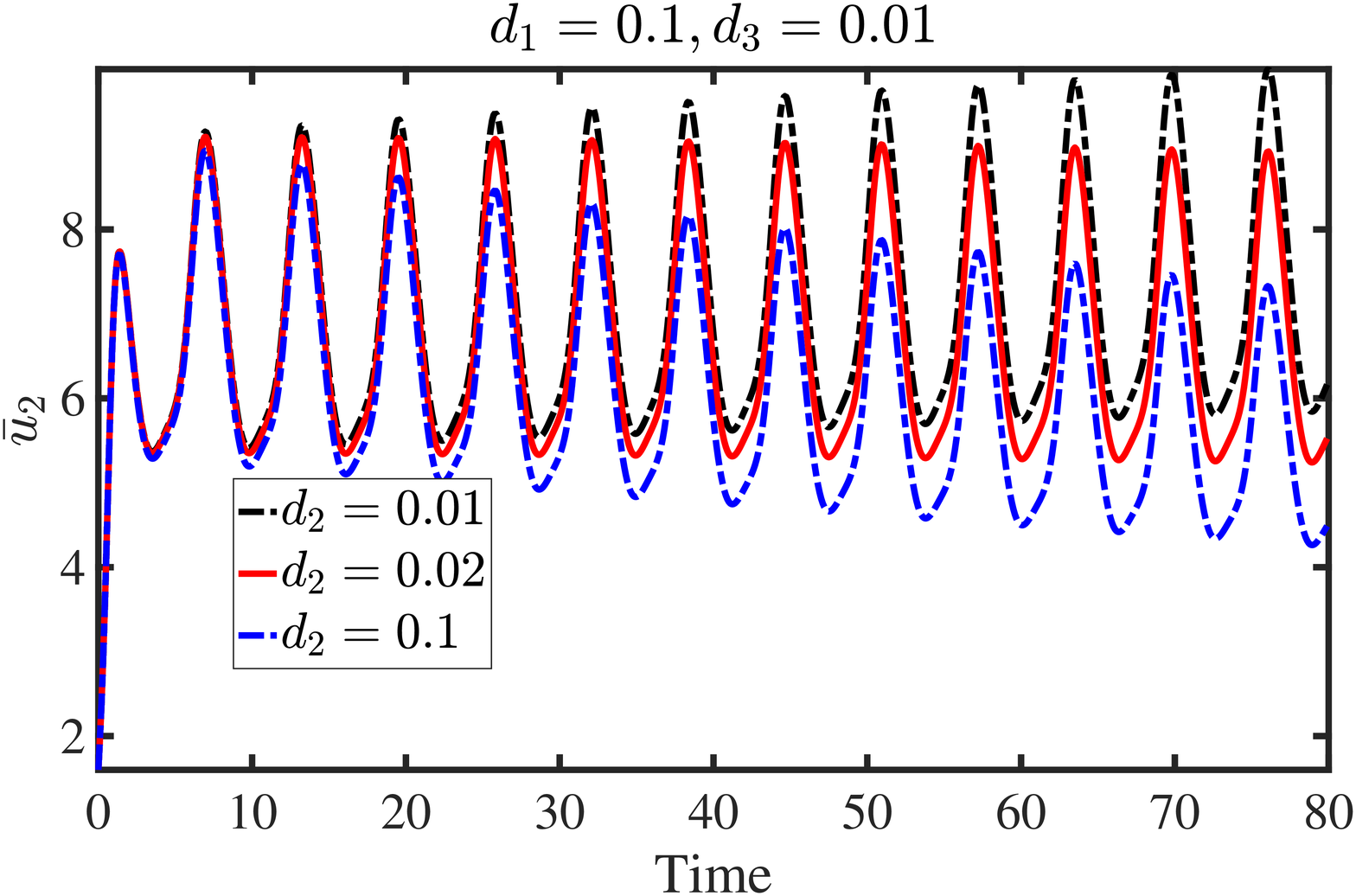}}\\
		\subfloat[]{\includegraphics[width=0.5\textwidth,height=0.3\textwidth]{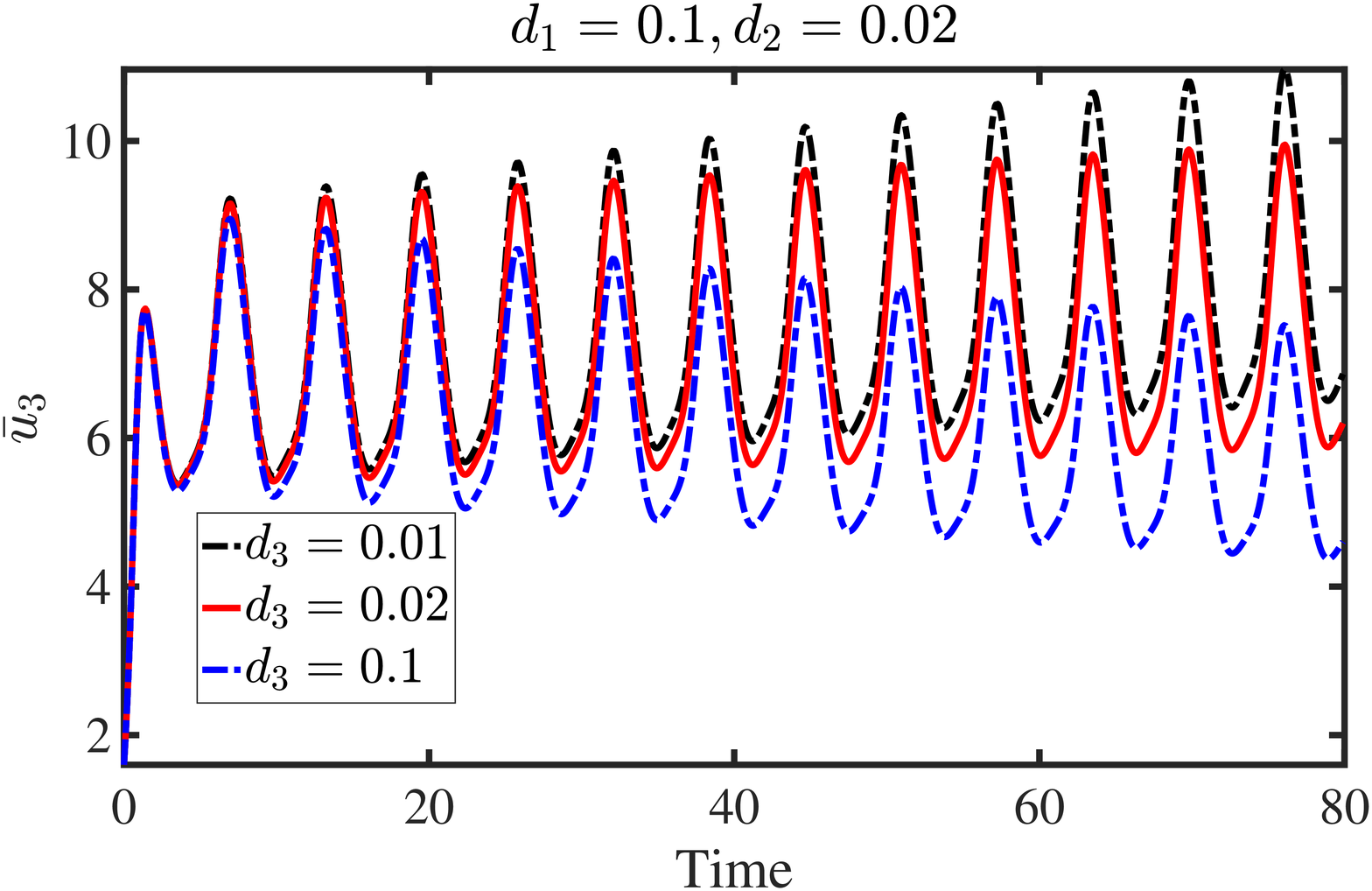}}
		\subfloat[]{\includegraphics[width=0.5\textwidth,height=0.3\textwidth]{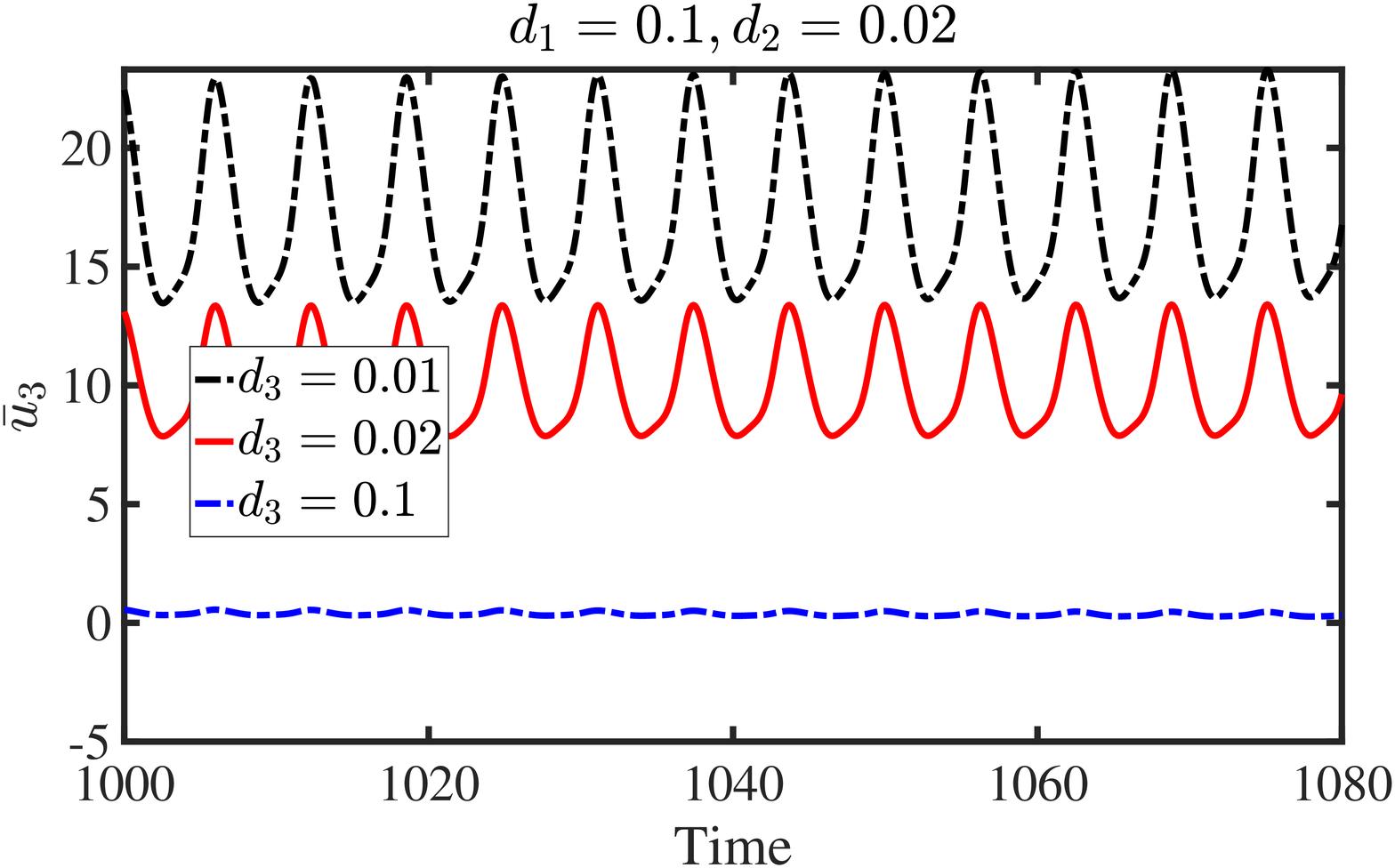}}
		\caption{The effect of diffusion rate on the average population density (a) $\bar{u}_1$, (b) $\bar{u}_2$, (c) $\bar{u}_3$, and (d) $\bar{u}_3$ on long-range without harvesting or stocking effort for $K(t,\bx)\equiv\big(1.2+2.5\pi^2e^{-(x-0.5)^2-(y-0.5)^2}\big)\big(1.0+0.3\cos(t)\big),\;
		\mu_i=0$ and $r_i(t,\bx)\equiv(1.5+\sin(x)\sin(y))(1.2+\sin(t)),\hspace{1mm}i=1,2,3.$}	\label{diffusion_varies}
	\end{figure}
	
\subsection{Test 4: Effect of harvesting on the evolution of population density}

In this experiment, we consider $N=3,$ a three species competition model with varying diffusion rates as $d_1=0.1$, $d_2=0.02$, and $d_3=0.01$. That is, the spreading rate of the first and third species is the highest and lowest, respectively. Consequently, without stocking, their population density at any time must be the lowest and highest, respectively as investigated and presented in Figure  \ref{diffusion_varies}.

In Figure \ref{varying-harvesting-para}, we plot the average density of each species versus time with varying harvesting parameters over the time interval $[1000,1080]$. We choose the time interval $[1000,1080]$ to exhibit the long-range behavior of the solutions. In Figure \ref{varying-harvesting-para} (a), we plot the long-range behavior of the average density of each of the species in absence of harvesting or stocking effort ($\mu_1=\mu_2=\mu_3=0.0$). We observe a periodic behavior in the average density of all the species, and the third species dominates the other species in competition. The periodic behavior is inherited from periodic system carrying capacity of the system. The lowest diffusion rate plays a key role for the third species in becoming the winner. For time periodic parameters in Figure \ref{varying-harvesting-para} (a), we notice the average density of $u_1$ approaches zero in an oscillatory fashion.

Figure \ref{varying-harvesting-para} (b) represents the average density of each of the three species on the time interval $[1000,1080]$, where only the third species is affected by harvesting with coefficient $\mu_3=0.001$. The sequential map presented the results based on the combined effects of harvesting and diffusion coefficient. Comparing Figure \ref{varying-harvesting-para} (a) and Figure \ref{varying-harvesting-para} (b), it is clear that due to the non-zero harvesting parameter $\mu_3=0.001$, the density of the third species has been reduced. On the other hand, clearly, the second species is also impacted by the harvesting of the third species. Because of the reduction in the population density of the third species, the other species get more resources to grow, and a significant boost is observed in the second species' density and a considerable amount of density increment is observed in the first species.

If we further increase the harvesting coefficients as $\mu_2=0.001$, and $\mu_3=0.002$, but keep $\mu_1=0$ (no harvesting) and plot the average density curves versus time for each species in Figure \ref{varying-harvesting-para} (c), we observe an evolutionary population density feature, especially for the first species. The harvesting in the second and third species provides an advantage to the first species, and thus an apparent co-existence of all three species is visible over the time $[1000,1080]$. It is noted that Figure \ref{varying-harvesting-para} reveals the effects of harvesting levels on the scaled average population density on periodic time-dependent functions as happens for seasonal changes. Though the considered values of the harvesting parameters are not corresponding to the optimal co-existence, it is possible to estimate optimal $(\mu_1,\mu_2,\mu_3)$ \cite{adan2022role}.
\begin{figure} [ht]
		\centering
		\subfloat[]{\includegraphics[scale=.17]{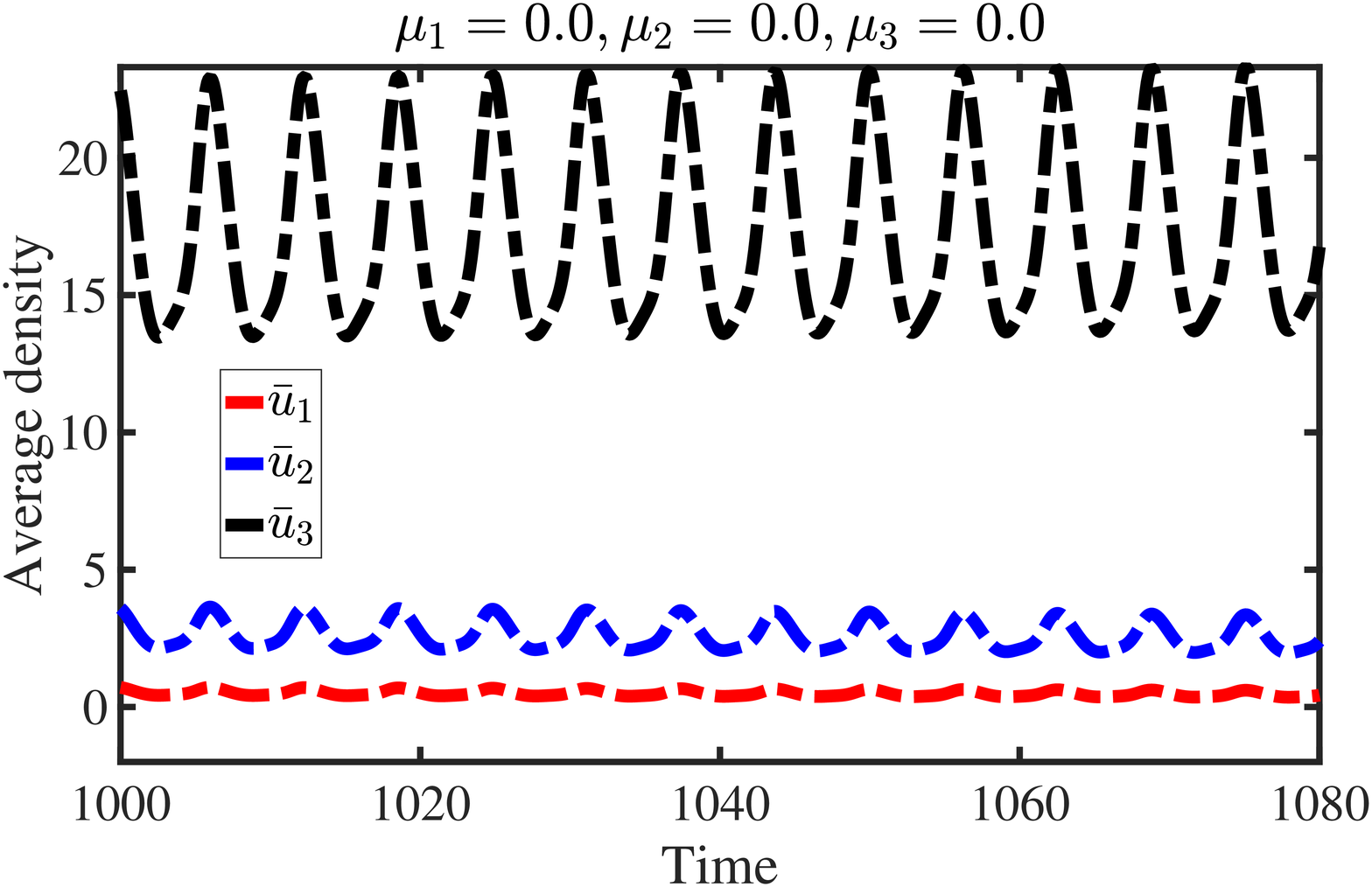}}
		\subfloat[]{\includegraphics[scale=.17]{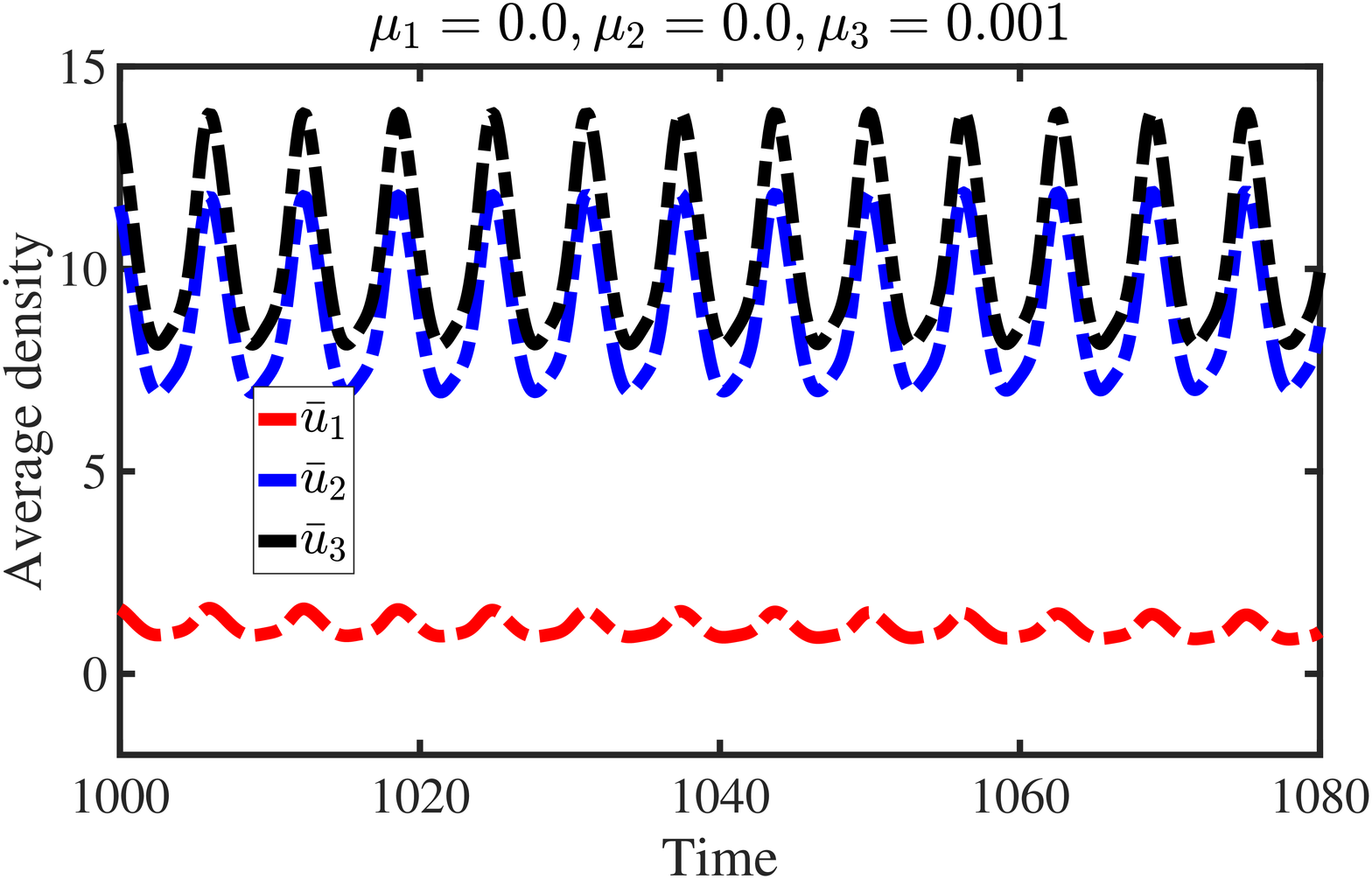}}
		\\
		\subfloat[]{\includegraphics[scale=.17]{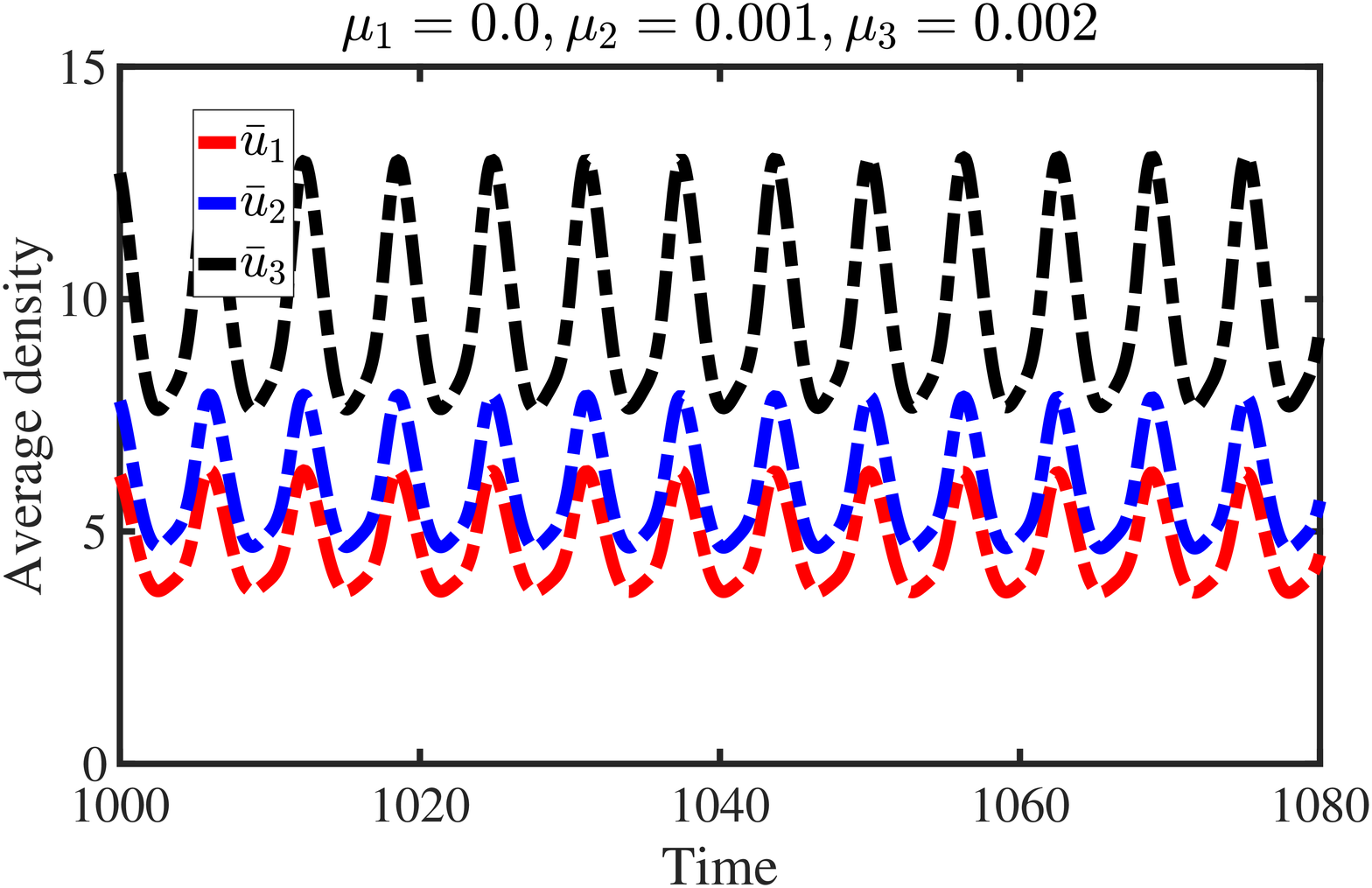}}
		\caption{The effect of harvesting parameters on the average density with diffusion parameters $d_1=0.1$, $d_2=0.02$, and $d_3=0.01$ for $K(t,\bx)\equiv\big(1.2+2.5\pi^2e^{-(x-0.5)^2-(y-0.5)^2}\big)\big(1.0+0.3\cos(t)\big),$ and $r_i(t,\bx)\equiv(1.5+\sin(x)\sin(y))(1.2+\sin(t)),\hspace{1mm}i=1,2,3.$ with (a) No harvesting, (b) third population harvested and (c) $u_2,\;u_3$ are harvesting.}\label{varying-harvesting-para}
	\end{figure}

\subsection{Test 5: Effect of stocking on the evolution of population density}

In this experiment, we observed how the evolution of population density is affected by the variation in stocking parameters. We plot the average density corresponding to each species versus time varying the stocking parameters in Figure \ref{varying-stocking-para} on the time interval $[1000,1080]$. Because of the periodic resource function, we observe periodic behavior in all the population densities (Figure \ref{varying-stocking-para}). We consider the same diffusion rates, $d_1=0.1$, $d_2=0.02$, and $d_3=0.01$ as in the case of Figure \ref{varying-harvesting-para} (a), where no harvesting or stocking is considered. 

\begin{figure} [ht]
		\centering
		\subfloat[]{\includegraphics[scale=.17]{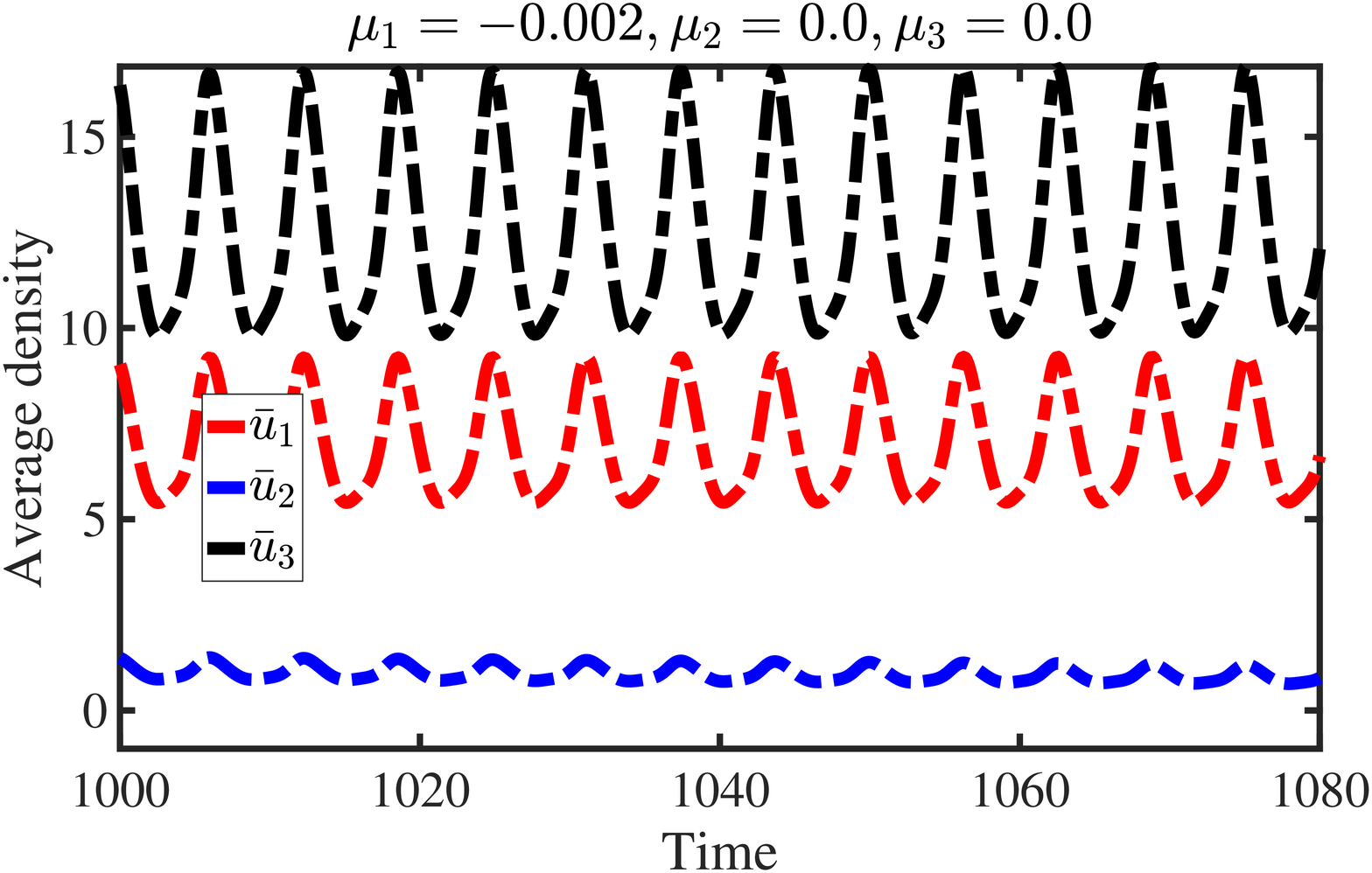}}
		\subfloat[]{\includegraphics[scale=.17]{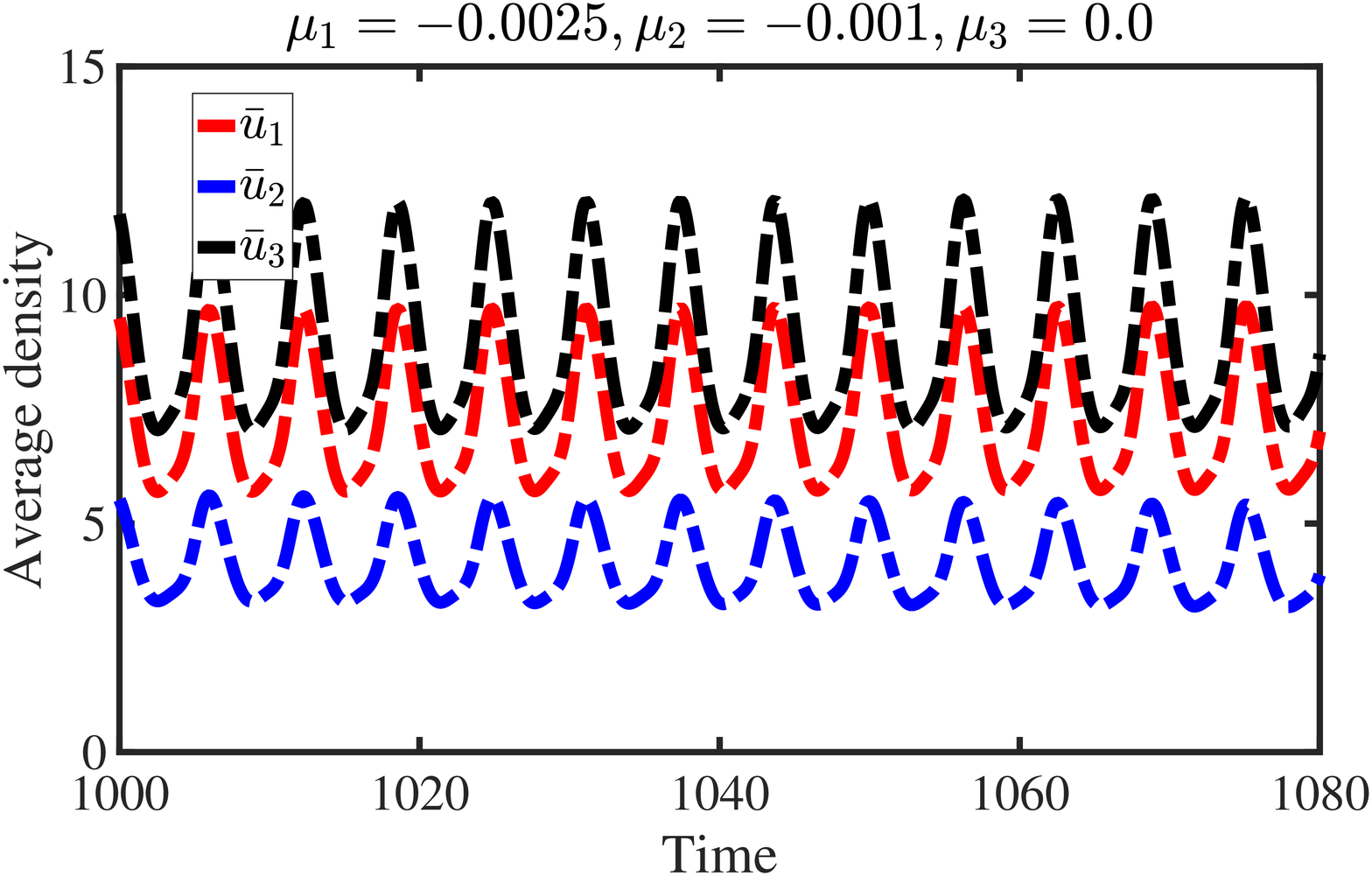}}
		\caption{Effect of stocking parameters on the system energy corresponding to each species density with the diffusion parameters $d_1=0.1$, $d_2=0.02$, and $d_3=0.01$ for $K(t,\bx)\equiv\big(1.2+2.5\pi^2e^{-(x-0.5)^2-(y-0.5)^2}\big)\big(1.0+0.3\cos(t)\big),$ and $r_i(t,\bx)\equiv(1.5+\sin(x)\sin(y))(1.2+\sin(t)),\hspace{1mm}i=1,2,3.$}\label{varying-stocking-para}
	\end{figure}

We reduce the  stocking parameter $\mu_1=0.0$ to $\mu_1=-0.002$ for $u_1$ since the diffusion rate is higher for the first species.
The results are displayed in Figure \ref{varying-stocking-para} (a). Comparing  Figure \ref{varying-harvesting-para} (a) and  Figure \ref{varying-stocking-para} (a), we observe that the population density of the first species increases while for the second and third species decrease. It provides the increased population of the first species consumes more resources from the environment, which reduces the productivity of the other two species.

Next, we decrease the stocking parameter of the first species to $\mu_1=-0.0025$ and for the second species to $\mu_2=-0.001$, keeping no harvesting or stocking to the third species, and plot the average density in Figure \ref{varying-stocking-para} (b). We observe that the density of the third species decreases while the first and second species increase, and there is a transparent co-existence of all the species. Therefore, the harvesting and stocking parameters can be a controlling tool in population dynamics to optimize the limited resources.

 
\section{Conclusion and Future Research}\label{conclusion}
Time evolutionary reaction-diffusion equation is the basis of harvesting and/or stocking model in population dynamics. In this paper, we propose a time-dependent system of non-linear coupled partial differential equations representing the dynamics of an $N$-species competition model with harvesting and/or stocking effect. We also propose, analyze and test two fully discrete decoupled stable algorithms for numerical computation. We have proven the first scheme is first-order accurate and the second scheme is second-order accurate in time and both are optimally accurate in space. Numerical tests verify the predicted convergence rates with some analytical test problems for both two- and three-species competition models. The linearized decoupled algorithms are efficient as at each time-step the solution for each species can be computed simultaneously, which can save a huge computational time in large-scale computing with complex problems.  

Numerical experiments exhibit (a) if the diffusion rate increases the population density decreases and faster the process of extinction, (b) if the harvesting parameter of a species increases, its density decreases, and other species get benefit in the competition, and (c) if the stocking effect of a species intensify, the population density increases and tends to win over the other species in the competition, that is, the survival period of the species increases. For a particular competition model, a set of values of the harvesting or stocking parameters can be found for which co-existence of all species will be ensured. The harvesting and/or stocking parameter can be a useful controlling tool in the population dynamics with limited natural resources.

As a future research, we will propose, analyze and test an efficient decoupled algorithm for this $N$-species competition model following the works in \cite{MR17,mohebujjaman2022efficient} so that at each time-step each species solver shares a common system matrix which will further save computer memory and assembling time of system matrix.

\bibliographystyle{plain}
\bibliography{Bib_Harvesting}
\section*{Appendix} \label{appendix} Here we
  find the restriction on the time-step size to have the stability of the Algorithm \ref{Algn1}. We consider the following linear system:
\begin{align}
    a(u_{i,h}^{n+1},v_{i,h})=F\left(v_{i,h}\right),\hspace{1mm}\forall v_{i,h}\in X_h,\hspace{1mm}i=1,2,\cdots,N,\label{linear-system}
\end{align}
where the linear form
\begin{align}
    F\left(v_{i,h}\right):=\frac{1}{\Delta t}\left(u_{i,h}^{n},v_{i,h}\right)+\left(f_i(t^{n+1}),v_{i,h}\right),
\end{align}
and  the  bilinear form

\begin{align}
    a\left(u_{i,h}^{n+1},v_{i,h}\right):&=\frac{1}{\Delta t}\left(u_{i,h}^{n+1},v_{i,h}\right)+d_i\left(\nabla u_{i,h}^{n+1},\nabla v_{i,h}\right)-(1-\mu_i)\left(r_i(t^{n+1})u_{i,h}^{n+1},v_{i,h}\right)\nonumber\\&-\left(\frac{r_i(t^{n+1})u_{i,h}^{n+1}}{K(t^{n+1})}\sum\limits_{j=1}^Nu_{j,h}^n,v_{i,h}
    \right).\label{bilinear-form}
\end{align}
Now, substitute $v_{i,h}=u_{i,h}^{n+1}$ in \eqref{bilinear-form} to give
\begin{align}
    a&\left(u_{i,h}^{n+1},u_{i,h}^{n+1}\right)=\frac{1}{\Delta t}\|u_{i,h}^{n+1}\|^2+d_i\|\nabla u_{i,h}^{n+1}\|^2-(1-\mu_i)\left(r_i(t^{n+1})u_{i,h}^{n+1},u_{i,h}^{n+1}\right)\nonumber\\&-\left(\frac{r_i(t^{n+1})u_{i,h}^{n+1}}{K(t^{n+1})}\sum\limits_{j=1}^Nu_{j,h}^n,u_{i,h}^{n+1}
    \right)\nonumber\\&\ge \frac{1}{\Delta t}\|u_{i,h}^{n+1}\|^2+d_i\|\nabla u_{i,h}^{n+1}\|^2-|1-\mu_i|\|r_i(t^{n+1})\|_{\infty}\|u_{i,h}^{n+1}\|^2-\frac{C\|r_i(t^{n+1})\|_{\infty}}{\inf\limits_\Omega\|K(t^{n+1})\|}\|u_{i,h}^{n+1}\|^2.
\end{align}
The last term in the above inequality is derived as the lower bound subject to the Assumption \ref{assumption-1}. Rearranging
\begin{align}
    a\left(u_{i,h}^{n+1},u_{i,h}^{n+1}\right)\ge\left(\frac{1}{\Delta t}-|1-\mu_i|\|r_i\|_{L^\infty\big(0,T;L^\infty(\Omega)^d\big)}-\frac{C\|r_i\|_{L^\infty\big(0,T;L^\infty(\Omega)^d\big)}}{\inf\limits_{(t,\bx)\in (0,T]\times\Omega}|K|}\right)\|u_{i,h}^{n+1}\|^2+d_i\|\nabla u_{i,h}^{n+1}\|^2.
\end{align}
To have the coercivity condition, we must have
\begin{align*}
    \frac{1}{\Delta t}-|1-\mu_i|\|r_i\|_{\infty,\infty}-\frac{C\|r_i\|_{\infty,\infty}}{K_{\min}}\ge 0,
\end{align*}
which gives the following restriction on the time-step size
\begin{align}
    \Delta t\le\frac{K_{\min}}{|1-\mu_i|\|r_i\|_{\infty,\infty}K_{\min}+C\|r_i\|_{\infty,\infty}}.\label{time-step-restriction}
\end{align}
In a similar approach, we can find the time-step restriction on the stability of the Algorithm \ref{Algn2}.
\end{document}